\documentclass[a4, 12pt]{amsart}
\usepackage{amssymb}
\usepackage{amstext}
\usepackage{amsmath}
\usepackage{amscd}
\usepackage{amsthm}
\usepackage{amsfonts}
\usepackage{enumerate}
\usepackage{graphicx}
\usepackage{latexsym}
\theoremstyle{plain}
\newtheorem{thm}{Theorem}[section]
\newtheorem*{thm*}{Theorem}
\newtheorem*{cor*}{Corollary}
\newtheorem*{prop*}{Proposition}

\newtheorem*{mthm}{Main Theorem}

\newtheorem{prop}[thm]{Proposition}
\newtheorem{lem}[thm]{Lemma}

\newtheorem{cor}[thm]{Corollary}
\newtheorem{claim}{Claim}
\newtheorem*{claim*}{Claim}

\theoremstyle{definition}
\newtheorem{dfn}[thm]{Definition}
\newtheorem{rem}[thm]{Remark}
\newtheorem{ex}[thm]{Example}

\newtheorem*{conj*}{Conjecture}

\newtheorem*{ac}{{\sc Acknowledgments}}

\theoremstyle{remark}

\newtheorem*{cpf}{{\it Proof of Claim}}

\numberwithin{equation}{thm}
\def\Hom{\operatorname{Hom}}
\def\lhom{\operatorname{\underline{Hom}}}

\def\Ext{\operatorname{Ext}}
\def\Tor{\operatorname{Tor}}

\def\mod{\operatorname{mod}}
\def\res{\operatorname{res}}

\def\add{\operatorname{add}}

\def\CM{\operatorname{CM}}
\def\lCM{\operatorname{\underline{CM}}}

\def\tr{\operatorname{Tr}}

\def\m{\mathfrak m}

\def\p{\mathfrak p}
\def\q{\mathfrak q}

\def\Ob{\operatorname{Ob}}

\def\depth{\operatorname{depth}}
\def\Supp{\operatorname{Supp}}

\def\lSupp{\operatorname{\underline{Supp}}}
\def\Ann{\operatorname{Ann}}
\def\Ass{\operatorname{Ass}}

\def\Min{\operatorname{Min}}

\def\pd{\operatorname{pd}}
\def\id{\operatorname{id}}

\def\hdim{\operatorname{{\bf H}-dim}}

\def\grade{\operatorname{grade}}

\def\Spec{\operatorname{Spec}}
\def\Sing{\operatorname{Sing}}
\def\Gdim{\operatorname{G-dim}}

\def\CMdim{\operatorname{CM-dim}}

\def\ext{\operatorname{ext}}

\def\A{{\mathcal A}}
\def\B{{\mathcal B}}

\def\F{{\mathcal F}}

\def\X{{\mathcal X}}
\def\Y{{\mathcal Y}}

\def\V{{\mathcal V}}
\def\v{\operatorname{V}}
\def\s{{\mathcal S}}

\def\ccc{{\mathcal C}}

\def\xx{\text{\boldmath $x$}}
\def\yy{\text{\boldmath $y$}}

\def\ng{\operatorname{nonGor}}

\begin{document}

\setlength{\baselineskip}{15pt}

%%%%%%%%%%%%%%%%%%%%%%%%%%%%%%%%%%%%%%%%%%%%%%%%%%%%%%%%%%%%%
%%%%%%%%%%%%%%%%%%%%%%%%%%%%%%%%%%%%%%%%%%%%%%%%%%%%%%%%%%%%%
\title[Thick subcategories of the stable category]{Classifying thick subcategories of the stable category of Cohen-Macaulay modules}
\author{Ryo Takahashi}
\address{Department of Mathematical Sciences, Faculty of Science, Shinshu University, 3-1-1 Asahi, Matsumoto, Nagano 390-8621, Japan}
\email{takahasi@math.shinshu-u.ac.jp}
\keywords{thick subcategory, triangulated category, stable category, resolving subcategory, support, specialization-closed subset, nonfree locus, Cohen-Macaulay module, Cohen-Macaulay ring, Gorenstein ring, hypersurface}
%\dedicatory{Dedicated to Professor Shiro Goto on the occasion of his sixtieth birthday}
\subjclass[2000]{13C05, 13C14, 13H10, 16D90, 16G60, 18E30}
\begin{abstract}
Various classification theorems of thick subcategories of a triangulated category have been obtained in many areas of mathematics.
In this paper, as a higher-dimensional version of the classification theorem of thick subcategories of the stable category of finitely generated representations of a finite $p$-group due to Benson, Carlson and Rickard, we consider classifying thick subcategories of the stable category of Cohen-Macaulay modules over a Gorenstein local ring.
The main result of this paper yields a complete classification of the thick subcategories of the stable category of Cohen-Macaulay modules over a local hypersurface in terms of specialization-closed subsets of the prime ideal spectrum of the ring which are contained in its singular locus.
\end{abstract}
\maketitle
%%%%%%%%%%%%%%%%%%%%%%%%%%%%%%%%%%%%%%%%%%%%%%%%%%%%%%%%%%%%%
%%%%%%%%%%%%%%%%%%%%%%%%%%%%%%%%%%%%%%%%%%%%%%%%%%%%%%%%%%%%%
%\tableofcontents
%%%%%%%%%%%%%%%%%%%%%%%%%%%%%%%%%%%%%%%%%%%%%%%%%%%%%%%%%%%%%
%%%%%%%%%%%%%%%%%%%%%%%%%%%%%%%%%%%%%%%%%%%%%%%%%%%%%%%%%%%%%
\section*{Introduction}\label{intro}
%%%%%%%%%%%%%%%%%%%%%%%%%%%%%%%%%%%%%%%%%%%%%%%%%%%%%%%%%%%%%
%%%%%%%%%%%%%%%%%%%%%%%%%%%%%%%%%%%%%%%%%%%%%%%%%%%%%%%%%%%%%

One of the principal approaches to the understanding of the structure of a given category is classifying its subcategories having a specific property.
It has been studied in many areas of mathematics which include stable homotopy theory, ring theory, algebraic geometry and modular representation theory.
A landmark result in this context was obtained in the definitive work due to Gabriel \cite{G} in the early 1960s.
He proved a classification theorem of the localizing subcategories of the category of modules over a commutative noetherian ring by making a one-to-one correspondence between the set of those subcategories and the set of specialization-closed subsets of the prime ideal spectrum of the ring.
A lot of analogous classification results of subcategories of modules have been obtained by many authors; see \cite{Hov,wide,K,GP,GP2,GP3} for instance.

For a triangulated category, a high emphasis has been placed on classifying its {\em thick} subcategories, namely, full triangulated subcategories which are closed under taking direct summands.
The first classification theorem was obtained in the deep work on stable homotopy theory due to Devinatz, Hopkins and Smith \cite{DHS,HS}.
They classified the thick subcategories of the category of compact objects in the $p$-local stable homotopy category.
Hopkins \cite{Hop} and Neeman \cite{N} provided a corresponding classification result of the thick subcategories of the derived category of perfect complexes (i.e., bounded complexes of finitely generated projective modules) over a commutative noetherian ring by making a one-to-one correspondence between the set of those subcategories and the set of specialization-closed subsets of the prime ideal spectrum of the ring.
Thomason \cite{To} generalized the theorem of Hopkins and Neeman to quasi-compact and quasi-separated schemes, in particular, to arbitrary commutative rings and algebraic varieties.
Recently, Avramov, Buchweitz, Christensen, Iyengar and Piepmeyer \cite{ABCIP} gave a classification of the thick subcategories of the derived category of perfect differential modules over a commutative noetherian ring.
On the other hand, Benson, Carlson and Rickard \cite{BCR} classified the thick subcategories of the stable category of finitely generated representations of a finite $p$-group in terms of closed homogeneous subvarieties of the maximal ideal spectrum of the group cohomology ring.
Friedlander and Pevtsova \cite{FP} extended this classification theorem to finite group schemes.
A recent work of Benson, Iyengar and Krause \cite{BIK2} gives a new proof of the theorem of Benson, Carlson and Rickard.
A lot of other related results concerning thick subcategories of a triangulated category have been obtained.
For example, see \cite{B3,B,B2,K2,BKS,BIK,I,Br,DGI,hcls}.

Here we mention that in most of the classification theorems of subcategories stated above, the subcategories are classified in terms of certain sets of prime ideals.
Each of them establishes an assignment corresponding each subcategory to a set of prime ideals, which is (or should be) called the {\em support} of the subcategory.

In the present paper, as a higher-dimensional version of the work of Benson, Carlson and Rickard, we consider classifying thick subcategories of the stable category of Cohen-Macaulay modules over a Gorenstein local ring, through defining a suitable support for those subcategories.
Over a hypersurface we shall give a complete classification of them in terms of specialization-closed subsets of the prime ideal spectrum of the base ring contained in its singular locus.
To state in the following our results precisely, we introduce some notation.
Let $R$ be a (commutative) Cohen-Macaulay local ring.
We denote by $\mod R$ the category of finitely generated $R$-modules, by $\CM(R)$ the full subcategory of $\mod R$ consisting of all (maximal) Cohen-Macaulay $R$-modules, and by $\lCM(R)$ the stable category of $\CM(R)$.
It is a well-known result due to Buchweitz \cite{Bu} that $\lCM(R)$ is a triangulated category if $R$ is Gorenstein.
Let $\Spec R$ denote the prime ideal spectrum of $R$, that is, the set of prime ideals of $R$, and let $\Sing R$ denote the singular locus of $R$, that is, the set of prime ideals $\p$ of $R$ such that the local ring $R_\p$ is singular.
We denote the $n^\text{th}$ syzygy of an $R$-module $M$ by $\Omega^nM$.
The main result of this paper is the following theorem, which is part of Theorem \ref{main}.

\begin{mthm}
\begin{enumerate}[\rm (1)]
\item
Let $R$ be an abstract hypersurface local ring (i.e., the completion of $R$ is isomorphic to $S/(f)$ for some complete regular local ring $S$ and some element $f$ of $S$).
Then one has the following one-to-one correspondences:
$$
\begin{CD}
\{\text{thick subcategories of }\lCM(R)\}\\
\begin{matrix}
@V{\lSupp}VV @AA{\lSupp^{-1}}A
\end{matrix}
\\
\{\text{specialization-closed subsets of }\Spec R\text{ contained in }\Sing R\}\\
\begin{matrix}
@V{\V_{\CM}^{-1}}VV @AA{\V}A \\
\end{matrix}
\\
\{\text{resolving subcategories of }\mod R\text{ contained in }\CM(R)\}.
\end{CD}
$$
\item
Let $R$ be a $d$-dimensional Gorenstein singular local ring with residue field $k$ which is locally an abstract hypersurface on the punctured spectrum.
Then one has the following one-to-one correspondences:
$$
\begin{CD}
\{\text{thick subcategories of }\lCM(R)\text{ containing }\Omega^dk\}\\
\begin{matrix}
@V{\lSupp}VV @AA{\lSupp^{-1}}A
\end{matrix}
\\
\{\text{nonempty specialization-closed subsets of }\Spec R\text{ contained in }\Sing R\}\\
\begin{matrix}
@V{\V_{\CM}^{-1}}VV @AA{\V}A \\
\end{matrix}
\\
\{\text{resolving subcategories of }\mod R\text{ contained in }\CM(R)\text{ containing }\Omega^dk\}.
\end{CD}
$$
\end{enumerate}
\end{mthm}

In this theorem, a {\em resolving subcategory} means a full subcategory containing the free modules which is closed under taking direct summands, extensions and syzygies.
The notion of a resolving subcategory was introduced by Auslander and Bridger \cite{AB} in the late 1960s, and a lot of important resolving subcategories of modules are known; see Example \ref{resex}.
The symbol $\lSupp$, which we call the {\em stable support}, is a support for the stable category of Cohen-Macaulay modules.
It is defined quite similarly to the ordinary support for modules; the stable support of an object $M$ of $\lCM(R)$ is the set of prime ideals $\p$ of $R$ such that the localization $M_\p$ is not isomorphic to $0$ in $\lCM(R_\p)$, and the stable support of a subcategory of $\lCM(R)$ is the union of the stable supports of all (nonisomorphic) objects in it.
The symbol $\V$ denotes the {\em nonfree locus}; the nonfree locus of an object $M$ of $\mod R$ is the set of prime ideals $\p$ such that $M_\p$ is nonfree as an $R_\p$-module, and the nonfree locus of a subcategory of $\mod R$ is the union of the nonfree loci of all objects in it.
The inverse maps are defined as follows: for a subset $\Phi$ of $\Spec R$, $\lSupp^{-1}(\Phi)$ (respectively, $\V_{\CM}^{-1}(\Phi)$) is the full subcategory of $\lCM(R)$ (respectively, $\CM(R)$) consisting of all objects whose stable supports (respectively, nonfree loci) are contained in $\Phi$.

The above theorem especially enables us to reconstruct each thick subcategory of $\lCM(R)$ (respectively, resolving subcategory contained in $\CM(R)$) from its support (respectively, nonfree locus).
We will actually prove more general statements than the above theorem; see Theorems \ref{origin}, \ref{lochp}, \ref{tilde} and \ref{main}.

Very recently, after the work in this paper was completed, Iyengar announced in his lecture \cite{I2} that thick subcategories of the bounded derived category of finitely generated modules over a locally complete intersection which is essentially of finite type over a field are classified in terms of certain subsets of the prime ideal spectrum of the Hochschild cohomology ring.
This provides a classification of thick subcategories of the stable category of Cohen-Macaulay modules over such a ring, which is a different classification from ours.

%%%%%%%%%%%%%%%%%%%%%%%%%%%%%%%%%%%%%%%%%%%%%%%%%%%%%%%%%%%%%
%%%%%%%%%%%%%%%%%%%%%%%%%%%%%%%%%%%%%%%%%%%%%%%%%%%%%%%%%%%%%

\section*{Convention}\label{convent}
In the rest of this paper, we assume that all rings are commutative and noetherian, and that all modules are finitely generated.
Unless otherwise specified, let $R$ be a local ring of Krull dimension $d$.
The unique maximal ideal of $R$ and the residue field of $R$ are denoted by $\m$ and $k$, respectively.
By a {\em subcategory}, we always mean a full subcategory which is closed under isomorphism.
(A full subcategory $\X$ of a category $\ccc$ is said to be closed under isomorphism provided that for two objects $M,N$ of $\ccc$ if $M$ belongs to $\X$ and $N$ is isomorphic to $M$ in $\ccc$, then $N$ also belongs to $\X$.)
Note that a subcategory in our sense is uniquely determined by the isomorphism classes of the objects in it.

%%%%%%%%%%%%%%%%%%%%%%%%%%%%%%%%%%%%%%%%%%%%%%%%%%%%%%%%%%%%%
%%%%%%%%%%%%%%%%%%%%%%%%%%%%%%%%%%%%%%%%%%%%%%%%%%%%%%%%%%%%%
\section{Preliminaries}\label{prelim}
%%%%%%%%%%%%%%%%%%%%%%%%%%%%%%%%%%%%%%%%%%%%%%%%%%%%%%%%%%%%%
%%%%%%%%%%%%%%%%%%%%%%%%%%%%%%%%%%%%%%%%%%%%%%%%%%%%%%%%%%%%%

In this section, we give several basic notions and results which will be used later.
We begin with recalling the definitions of the syzygies and the transpose of a module.

\begin{dfn}
Let $n$ be a nonnegative integer, and let $M$ be an $R$-module.
Let
$$
\cdots \overset{\partial_{n+1}}{\to} F_n \overset{\partial_n}{\to} F_{n-1} \overset{\partial_{n-1}}{\to} \cdots \overset{\partial_2}{\to} F_1 \overset{\partial_1}{\to} F_0 \to M \to 0
$$
be a minimal free resolution of $M$.
\begin{enumerate}[(1)]
\item
The {\em $n^\text{th}$ syzygy} of $M$ is defined as the image of the map $\partial_n$, and we denote it by $\Omega^nM$ (or $\Omega_R^nM$ when there is some fear of confusion).
We simply write $\Omega M$ instead of $\Omega^1M$.
Note that the $n^\text{th}$ syzygy of a given $R$-module is uniquely determined up to isomorphism because so is a minimal free resolution.
\item
The {\em (Auslander) transpose} of $M$ is defined as the cokernel of the map $\Hom_R(\partial_1,R):\Hom_R(F_0,R)\to\Hom_R(F_1,R)$, and we denote it by $\tr M$ (or $\tr_RM$).
Note that the transpose of a given $R$-module is uniquely determined up to isomorphism because so is a minimal free resolution.
\end{enumerate}
\end{dfn}

Next, we make a list of several closed properties of a subcategory.

\begin{dfn}
\begin{enumerate}[(1)]
\item
Let $\ccc$ be an additive category and $\X$ a subcategory of $\ccc$.
\begin{enumerate}[(i)]
\item
We say that $\X$ is {\em closed under (finite) direct sums} provided that if $M_1,\dots,M_n$ are objects of $\X$, then the direct sum $M_1\oplus\cdots\oplus M_n$ in $\ccc$ belongs to $\X$.
\item
We say that $\X$ is {\em closed under direct summands} provided that if $M$ is an object of $\X$ and $N$ is a direct summand of $M$ in $\ccc$, then $N$ belongs to $\X$.
\end{enumerate}
\item
Let $\ccc$ be a triangulated category and $\X$ a subcategory of $\ccc$.
We say that $\X$ is {\em closed under triangles} provided that for each exact triangle $L \to M \to N \to \Sigma L$ in $\ccc$, if two of $L,M,N$ belong to $\X$, then so does the third.
\item
We denote by $\mod R$ the category of (finitely generated) $R$-modules.
Let $\X$ be a subcategory of $\mod R$.
\begin{enumerate}[(i)]
\item
We say that $\X$ is {\em closed under extensions} provided that for each exact sequence $0 \to L \to M \to N \to 0$ in $\mod R$, if $L$ and $N$ belong to $\X$, then so does $M$.
\item
We say that $\X$ is {\em closed under kernels of epimorphisms} provided that for each exact sequence $0 \to L \to M \to N \to 0$ in $\mod R$, if $M$ and $N$ belong to $\X$, then so does $L$.
\item
We say that $\X$ is {\em closed under syzygies} provided that if $M$ is an $R$-module in $\X$, then $\Omega^iM$ is also in $\X$ for all $i\ge 0$.
\end{enumerate}
\end{enumerate}
\end{dfn}

Let us make several definitions of subcategories.

\begin{dfn}
\begin{enumerate}[(1)]
\item
Let $\ccc$ be a category.
\begin{enumerate}[(i)]
\item
We call the subcategory of $\ccc$ which has no object the {\em empty subcategory} of $\ccc$.
\item
Suppose that $\ccc$ admits the zero object $0$.
We call the subcategory of $\ccc$ consisting of all objects that are isomorphic to $0$ the {\em zero subcategory} of $\ccc$.
\end{enumerate}
\item
A nonempty subcategory of a triangulated category is called {\em thick} if it is closed under direct summands and triangles.
\item
Let $\X$ be a subcategory of $\mod R$.
\begin{enumerate}[(i)]
\item
We say that $\X$ is {\em additively closed} if $\X$ is closed under direct sums and direct summands.
\item
We say that $\X$ is {\em extension-closed} if $\X$ is closed under direct summands and extensions.
\item
We say that $\X$ is {\em resolving} if $\X$ contains $R$, and if $\X$ is closed under direct summands, extensions and kernels of epimorphisms.
\end{enumerate}
\end{enumerate}
\end{dfn}

\begin{rem}
\begin{enumerate}[(1)]
\item
A resolving subcategory is a subcategory such that any two ``minimal'' resolutions of a module by modules in it have the same length; see \cite[Lemma (3.12)]{AB}.
\item
Let $\X$ be a subcategory of $\mod R$.
Then the following implications hold:
$$
\X\text{ is resolving}\quad\Rightarrow\quad\X\text{ is extension-closed}\quad\Rightarrow\quad\X\text{ is additively closed}.
$$
\item
Every resolving subcategory of $\mod R$ contains all free $R$-modules.
\end{enumerate}
\end{rem}

In the definition of a resolving subcategory, closedness under kernels of epimorphisms can be replaced with closedness under syzygies.

\begin{prop}\cite[Lemma 3.2]{Y}\label{resdef}
A subcategory $\X$ of $\mod R$ is resolving if and only if $\X$ contains $R$ and is closed under direct summands, extensions and syzygies.
\end{prop}

A lot of important subcategories of $\mod R$ are known to be resolving.
To present examples of a resolving subcategory, let us recall here several definitions of modules.
Let $M$ be an $R$-module.
We say that $M$ is {\em bounded} if there exists an integer $s$ such that $\beta_i^R(M)\le s$ for all $i\ge 0$, where $\beta_i^R(M)$ denotes the $i^\text{th}$ Betti number of $M$.
We say that $M$ has {\em complexity} $c$ if $c$ is the least nonnegative integer $n$ such that there exists a real number $r$ satisfying the inequality $\beta_i^R(M)\le ri^{n-1}$ for $i\gg 0$.
We call $M$ {\em semidualizing} if the natural homomorphism $R\to\Hom_R(M,M)$ is an isomorphism and $\Ext_R^i(M,M)=0$ for all $i>0$.
For a semidualizing $R$-module $C$, an $R$-module $M$ is called {\em totally $C$-reflexive} if the natural homomorphism $M\to\Hom_R(\Hom_R(M,C),C)$ is an isomorphism and $\Ext_R^i(M,C)=0=\Ext_R^i(\Hom_R(M,C),C)$ for all $i>0$.
A totally $R$-reflexive $R$-module is simply called a {\em totally reflexive} $R$-module.
We say that $M$ has {\em lower complete intersection dimension zero} if $M$ is totally reflexive and has finite complexity.
When $R$ is a Cohen-Macaulay local ring, we say that $M$ is {\em Cohen-Macaulay} if $\depth M=d$.
Such a module is usually called maximal Cohen-Macaulay, but in this paper, we call it just Cohen-Macaulay.
We denote by $\CM(R)$ the subcategory of $\mod R$ consisting of all Cohen-Macaulay $R$-modules.

\begin{ex}\label{resex}
Let $n$ be a nonnegative integer, $K$ an $R$-module, and $I$ an ideal of $R$.
The following $R$-modules form resolving subcategories of $\mod R$.
\begin{enumerate}[(1)]
\item
The $R$-modules.
\item
The free $R$-modules.
\item
The Cohen-Macaulay $R$-modules, provided that $R$ is Cohen-Macaulay.
\item
The totally $C$-reflexive $R$-modules, where $C$ is a fixed semidualizing $R$-module.
\item
The $R$-modules $M$ with $\Tor_i^R(M,K)=0$ for $i>n$ (respectively, $i\gg 0$).
\item
The $R$-modules $M$ with $\Ext_R^i(M,K)=0$ for $i>n$ (respectively, $i\gg 0$).
\item
The $R$-modules $M$ with $\Ext_R^i(K,M)=0$ for $i\gg 0$, provided that $\Ext_R^j(K,R)=0$ for $j\gg 0$.
\item
The $R$-modules $M$ with $\Ext_R^i(K,M)=0$ for $i<\grade K(:=\grade(\Ann K,R))$.
\item
The $R$-modules $M$ with $\grade(I,M)\ge\grade(I,R)$.
\item
The bounded $R$-modules.
\item
The $R$-modules having finite complexity.
\item
The $R$-modules of lower complete intersection dimension zero.
\end{enumerate}
Each of the resolving properties of these subcategories can be verified either straightforwardly or by referring to \cite[Example 2.4]{res}.
\end{ex}

Let $\ccc$ be a category, and let ${\bf P}$ be a property of subcategories of $\ccc$.
Let $\X$ be a subcategory of $\ccc$.
A subcategory $\Y$ of $\ccc$ satisfying ${\bf P}$ is said to be {\em generated by} $\X$ if $\Y$ is the smallest subcategory of $\ccc$ satisfying ${\bf P}$ that contains $\X$.

Now we state the definitions of the closures corresponding to an additively closed subcategory, an extension-closed subcategory and a resolving subcategory of $\mod R$.

\begin{dfn}
Let $\X$ be a subcategory of $\mod R$.
\begin{enumerate}[(1)]
\item
We denote by $\add\X$ (or $\add_R\X$ when there is some fear of confusion) the additively closed subcategory of $\mod R$ generated by $\X$, and call it the {\em additive closure} of $\X$.
If $\X$ consists of a single module $X$, then we simply write $\add X$ (or $\add_RX$) instead of $\add\X$.
\item
We denote by $\ext\X$ (or $\ext_R\X$) the extension-closed subcategory of $\mod R$ generated by $\X$, and call it the {\em extension closure} of $\X$.
If $\X$ consists of a single module $X$, then we simply write $\ext X$ (or $\ext_RX$) instead of $\ext\X$.
\item
We denote by $\res\X$ (or $\res_R\X$) the resolving subcategory of $\mod R$ generated by $\X$, and call it the {\em resolving closure} of $\X$.
If $\X$ consists of a single module $X$, then we simply write $\res X$ (or $\res_RX$) instead of $\res\X$.
\end{enumerate}
\end{dfn}

We easily observe that $\add\X$ coincides with the subcategory of $\mod R$ consisting of all direct summands of direct sums of modules in $\X$.
The additive and extension closures of a given subcategory of $\mod R$ can be constructed inductively.

\begin{dfn}
Let $\X$ be a subcategory of $\mod R$ and $n$ a nonnegative integer.
\begin{enumerate}[(1)]
\item
We inductively define the subcategory $\ext^n\X$ (or $\ext_R^n\X$ when there is some fear of confusion) of $\mod R$ as follows.
\begin{enumerate}[(i)]
\item
Set $\ext^0\X=\add\X$.
\item
For $n\ge 1$, let $\ext^n\X$ be the additive closure of the subcategory of $\mod R$ consisting of all $R$-modules $Y$ having an exact sequence of the form
$$
0 \to A \to Y \to B \to 0,
$$
where $A,B\in\ext^{n-1}\X$.
\end{enumerate}
If $\X$ consists of a single module $X$, then we simply write $\ext^nX$ (or $\ext_R^nX$) instead of $\ext^n\X$.
\item
We inductively define a subcategory $\res^n\X$ (or $\res_R^n\X$) of $\mod R$ as follows.
\begin{enumerate}[(i)]
\item
Let $\res^0\X$ be the additive closure of the subcategory of $\mod R$ consisting of all modules in $\X$ and $R$.
\item
For $n\ge 1$, let $\res^n\X$ be the additive closure of the subcategory of $\mod R$ consisting of all $R$-modules $Y$ having an exact sequence of either of the following two forms:
\begin{align*}
& 0 \to A \to Y \to B \to 0,\\
& 0 \to Y \to A \to B \to 0,
\end{align*}
where $A,B\in\res^{n-1}\X$.
\end{enumerate}
If $\X$ consists of a single module $X$, then we simply write $\res^nX$ (or $\res_R^nX)$ instead of $\res^n\X$.
\end{enumerate}
\end{dfn}

\begin{rem}
Let $\X,\Y$ be subcategories of $\mod R$, and let $n$ be a nonnegative integer.
\begin{enumerate}[(1)]
\item
\begin{enumerate}[(i)]
\item
There is an ascending chain $\ext^0\X\subseteq\ext^1\X\subseteq\cdots\subseteq\ext^n\X\subseteq\cdots\subseteq\ext\X$ of subcategories of $\mod R$.
\item
The equality $\ext\X=\bigcup_{n\ge 0}\ext^n\X$ holds.
\end{enumerate}
\item
\begin{enumerate}[(i)]
\item
There is an ascending chain $\res^0\X\subseteq\res^1\X\subseteq\cdots\subseteq\res^n\X\subseteq\cdots\subseteq\res\X$ of subcategories of $\mod R$.
\item
The equality $\res\X=\bigcup_{n\ge 0}\res^n\X$ holds.
\end{enumerate}
\end{enumerate}
\end{rem}

The associated primes of modules in the additive and extension closures of a module are restricted.

\begin{prop}\label{ass}
Let $X$ and $M$ be $R$-modules.
\begin{enumerate}[\rm (1)]
\item
If $M$ is in $\ext X$, then one has $\Ass M\subseteq\Ass X$.
\item
If $M$ is in $\res X$, then one has $\Ass M\subseteq\Ass X\cup\Ass R$.
\end{enumerate}
\end{prop}

\begin{proof}
If $M$ is in $\ext X$ (respectively, $\res X$), then $M$ is in $\ext^nX$ (respectively, $\res^nX$) for some $n\ge 0$.
One easily deduces the conclusion by induction on $n$.
\end{proof}

For a subcategory $\X$ of $\mod R$ and a prime ideal $\p$ of $R$, we denote by $\X_\p$ the subcategory of $\mod R_\p$ consisting of $X_\p$ where $X$ runs through all modules in $\X$.
The proposition below gives the relationship between each closure and localization.

\begin{prop}\label{clloc}
Let $\X$ be a subcategory of $\mod R$, and let $\p$ be a prime ideal of $R$.
Then the following hold.
\begin{enumerate}[\rm (1)]
\item
$(\add_R\X)_\p$ is contained in $\add_{R_\p}\X_\p$.
\item
$(\ext_R\X)_\p$ is contained in $\ext_{R_\p}\X_\p$.
\item
$(\res_R\X)_\p$ is contained in $\res_{R_\p}\X_\p$.
\end{enumerate}
\end{prop}

\begin{proof}
The first statement is obvious.
The third statement is proved in \cite[Proposition 3.5]{res}, and an analogous argument shows the second statement.
\end{proof}

There are lower bounds for the depths of modules in the extension and resolving closures of a module.

\begin{prop}\label{cldep}
Let $X$ and $M$ be $R$-modules.
\begin{enumerate}[\rm (1)]
\item
If $M$ is in $\ext X$, then $\depth M\ge\depth X$.
\item
If $M$ is in $\res X$, then $\depth M\ge\inf\{\depth X,\depth R\}$.
\end{enumerate}
\end{prop}

This is proved by induction similarly to Proposition \ref{ass}.

Next we recall the definitions of the nonfree loci of an $R$-module and a subcategory of $\mod R$.

\begin{dfn}\label{defnf}
\begin{enumerate}[(1)]
\item
We denote by $\V(X)$ (or $\V_R(X)$ when there is some fear of confusion) the {\em nonfree locus} of an $R$-module $X$, namely, the set of prime ideals $\p$ of $R$ such that $X_\p$ is nonfree as an $R_\p$-module.
\item
We denote by $\V(\X)$ (or $\V_R(\X)$) the {\em nonfree locus} of a subcategory $\X$ of $\mod R$, namely, the union of $\V(X)$ where $X$ runs through all nonisomorphic $R$-modules in $\X$.
\end{enumerate}
\end{dfn}

We denote by $\Sing R$ the {\em singular locus} of $R$, namely, the set of prime ideals $\p$ of $R$ such that $R_\p$ is not a regular local ring.
We denote by $\s(R)$ the set of prime ideals $\p$ of $R$ such that the local ring $R_\p$ is not a field.
Clearly, $\s(R)$ contains $\Sing R$.

For each ideal $I$ of $R$, we denote by $\v(I)$ the set of prime ideals of $R$ containing $I$.
Recall that a subset $Z$ of $\Spec R$ is called {\em specialization-closed} provided that if $\p\in Z$ and $\q\in\v(\p)$ then $\q\in Z$.
Note that every closed subset of $\Spec R$ is specialization-closed.

The proposition below gives some basic properties of nonfree loci.
The proofs of the assertions are stated in \cite[Example 2.9]{res}, \cite[Proposition 2.10 and Corollary 2.11]{res} and \cite[Corollary 3.6]{res}, respectively.
\begin{prop}\label{nfsupp}
\begin{enumerate}[\rm (1)]
\item
Let $R$ be a Cohen-Macaulay local ring.
Then the nonfree locus $\V(\CM(R))$ coincides with the singular locus $\Sing R$.
\item
One has $\V(X)=\Supp\Ext^1(X,\Omega X)$ for every $R$-module $X$.
In particular, the nonfree locus of an $R$-module is closed in $\Spec R$ in the Zariski topology.
The nonfree locus of a subcategory of $\mod R$ is not necessarily closed but at least specialization-closed in $\Spec R$, and is contained in $\s(R)$.
\item
One has $\V(\X)=\V(\add\X)=\V(\ext\X)=\V(\res\X)$ for every subcategory $\X$ of $\mod R$.
\end{enumerate}
\end{prop}

For a subset $\Phi$ of $\Spec R$, we denote by $\V^{-1}(\Phi)$ the subcategory of $\mod R$ consisting of all $R$-modules $M$ such that $\V(M)$ is contained in $\Phi$.
Here we make a list of several statements concerning nonfree loci.

\begin{prop}\label{variety}
\begin{enumerate}[\rm (1)]
\item
Let $N$ be a direct summand of an $R$-module $M$.
Then one has $\V(N)\subseteq\V(M)$.
\item
Let $0 \to L \to M \to N \to 0$ be an exact sequence of $R$-modules.
Then one has $\V(L)\subseteq\V(M)\cup\V(N)$ and $\V(M)\subseteq\V(L)\cup\V(N)$.
\item
For a subset $\Phi$ of $\Spec R$, the subcategory $\V^{-1}(\Phi)$ of $\mod R$ is resolving.
\item
For an ideal $I$ of $R$, one has $\V_R(R/I)=\v(I+(0:I))$.
\item
One has $\V_R(R/\p)=\v(\p)$ for every $\p\in\s(R)$.
\item
Let $\Phi$ be a specialization-closed subset of $\Spec R$ contained in $\s(R)$.
Then one has $R/\p\in\V^{-1}(\Phi)$ for every $\p\in\Phi$.
\end{enumerate}
\end{prop}

\begin{proof}
(1),(2) These are straightforward.

(3) As $\V(R)=\emptyset\subseteq\Phi$, we have $R\in\V^{-1}(\Phi)$.
The assertion follows from (1) and (2).

(4) Fix a prime ideal $\p$ of $R$.
The $R_\p$-module $R_\p/IR_\p$ is nonfree if and only if $IR_\p\ne R_\p$ and $IR_\p\ne 0$, if and only if $I\subseteq\p$ and $(0:I)\subseteq\p$.
Hence $\V_R(R/I)=\v(I+(0:I))$ holds.

(5) Since the local ring $R_\p$ is not a field, its maximal ideal $\p R_\p$ is nonzero.
Hence we see that $(0:\p)$ is contained in $\p$.
Then apply (4).

(6) This is a direct consequence of (5).
\end{proof}

We introduce the notion of a thick subcategory of the category of Cohen-Macaulay modules over a Cohen-Macaulay local ring.

\begin{dfn}
Let $R$ be a Cohen-Macaulay local ring.
Let $\X$ be a subcategory of $\CM(R)$ which is closed under direct summands.
We say that $\X$ is {\em thick} provided that for each short exact sequence $0 \to L \to M \to N \to 0$ of Cohen-Macaulay $R$-modules (we will often call such an exact sequence an {\em exact sequence in $\CM(R)$}), if two of $L,M,N$ belong to $\X$, then so does the third.
\end{dfn}

In the sense of Krause and Br\"{u}ning \cite{K,Br}, a thick subcategory of the category of all (not necessarily finitely generated) modules means a full subcategory which is closed under taking submodules, quotient modules, extensions and arbitrary direct sums.
Of course, this is a different notion from ours.
On the other hand, a thick subcategory of the category of modules (over a group algebra) in the sense of Benson, Iyengar and Krause \cite{BIK2} is similar to ours.

In our sense, a thick subcategory has two different meanings.
One of them is a subcategory of a triangulated category, and thick subcategories of the stable category of the category of Cohen-Macaulay modules over a Gorenstein local ring are main objects of this paper.
The other is what we defined just above, and thick subcategories of the category of Cohen-Macaulay modules over a Cohen-Macaulay local ring are also our main objects.
In fact, our two kinds of thick subcategory are essentially the same notion; we will observe it in Section \ref{tsoscm}.

We make below a list of several examples of a thick subcategories of Cohen-Macaulay modules.

\begin{ex}\label{thex}
Let $n$ be a nonnegative integer, and let $K$ be an $R$-module.
The following modules form thick subcategories of $\CM(R)$.
\begin{enumerate}[(1)]
\item
The Cohen-Macaulay $R$-modules.
\item
The free $R$-modules.
\item
The totally $C$-reflexive $R$-modules, where $C$ is a fixed semidualizing $R$-module.
\item
The Cohen-Macaulay $R$-modules $M$ with $\Tor_i^R(M,K)=0$ for $i\gg 0$.
\item
The Cohen-Macaulay $R$-modules $M$ with $\Ext_R^i(M,K)=0$ for $i\gg 0$.
\item
The Cohen-Macaulay $R$-modules $M$ with $\Ext_R^i(K,M)=0$ for $i\gg 0$.
\item
The bounded Cohen-Macaulay $R$-modules.
\item
The Cohen-Macaulay $R$-modules having finite complexity.
\item
The Cohen-Macaulay $R$-modules of lower complete intersection dimension zero.
\end{enumerate}
Each of the thick properties of these subcategories can be verified by using Example \ref{resex}.
\end{ex}

Now we recall the definition of the stable category of Cohen-Macaulay modules over a Cohen-Macaulay local ring.

\begin{dfn}
\begin{enumerate}[(1)]
\item
Let $M,N$ be $R$-modules.
We denote by $\F_R(M,N)$ the set of $R$-homomorphisms $M\to N$ factoring through free $R$-modules.
It is easy to observe that $\F_R(M,N)$ is an $R$-submodule of $\Hom_R(M,N)$.
We set $\lhom_R(M,N)=\Hom_R(M,N)/\F_R(M,N)$.
\item
Let $R$ be a Cohen-Macaulay local ring.
The {\em stable category} of $\CM(R)$, which is denoted by $\lCM(R)$, is defined as follows.
\begin{enumerate}[(i)]
\item
$\Ob(\lCM(R))=\Ob(\CM(R))$.
\item
$\Hom_{\lCM(R)}(M,N)=\lhom_R(M,N)$ for $M,N\in\Ob(\lCM(R))$.
\end{enumerate}
\end{enumerate}
\end{dfn}

\begin{rem}
Let $R$ be a Cohen-Macaulay local ring.
Then $\lCM(R)$ is always an additive category.
The direct sum of objects $M$ and $N$ in $\lCM(R)$ is the direct sum $M\oplus N$ of $M$ and $N$ as $R$-modules.

From now on, we consider the case where $R$ is Gorenstein.
Then $\CM(R)$ is a Frobenius category, and $\lCM(R)$ is a triangulated category.
We recall in the following how to define an exact triangle in $\lCM(R)$.
For the details, we refer to \cite[Section 2 in Chapter I]{H} or \cite[Theorem 4.4.1]{Bu}.

Let $M$ be an object of $\lCM(R)$.
Then, since $M$ is a Cohen-Macaulay $R$-module, there exists an exact sequence $0 \to M \to F \to N \to 0$ of Cohen-Macaulay $R$-modules with $F$ free.
Defining $\Sigma M=N$, we have an automorphism $\Sigma:\lCM(R)\to\lCM(R)$ of categories.
This is the suspension functor.

Let
$$
\begin{CD}
0 @>>> L @>>> F @>>> \Sigma L @>>> 0 \\
@. @V{f}VV @VVV @| \\
0 @>>> M @>{g}>> N @>{h}>> \Sigma L @>>> 0
\end{CD}
$$
be a commutative diagram of Cohen-Macaulay $R$-modules with exact rows such that $F$ is free.
Then a sequence
$$
L' \overset{f'}{\to} M' \overset{g'}{\to} N' \overset{h'}{\to} \Sigma L'
$$
of morphisms in $\lCM(R)$ such that there exists a commutative diagram
$$
\begin{CD}
L @>{f}>> M @>{g}>> N @>{h}>> \Sigma L \\
@V{\alpha}VV @V{\beta}VV @V{\gamma}VV @V{\Sigma\alpha}VV \\
L' @>{f'}>> M' @>{g'}>> N' @>{h'}>> \Sigma L'
\end{CD}
$$
in $\lCM(R)$ such that $\alpha,\beta,\gamma$ are isomorphisms is defined to be an exact triangle in $\lCM(R)$.
\end{rem}

%%%%%%%%%%%%%%%%%%%%%%%%%%%%%%%%%%%%%%%%%%%%%%%%%%%%%%%%%%%%%
%%%%%%%%%%%%%%%%%%%%%%%%%%%%%%%%%%%%%%%%%%%%%%%%%%%%%%%%%%%%%
\section{Extension-closed subcategories of modules}\label{esom}
%%%%%%%%%%%%%%%%%%%%%%%%%%%%%%%%%%%%%%%%%%%%%%%%%%%%%%%%%%%%%
%%%%%%%%%%%%%%%%%%%%%%%%%%%%%%%%%%%%%%%%%%%%%%%%%%%%%%%%%%%%%

In this section, we study the extension closures of syzygies of the residue field over a Cohen-Macaulay local ring.
We begin with the following lemma, which follows from the proof of \cite[Lemma 2.2]{HP}.

\begin{lem}\label{syznzd}
Let $M$ be an $R$-module, and let $x$ be an $M$-regular element in $\Ann\Ext_R^1(M,\Omega M)$.
Then there is an isomorphism $\Omega_R(M/xM)\cong\Omega M\oplus M$ of $R$-modules.
\end{lem}

The following proposition is a generalization of \cite[Theorem 2.2]{OP}.
This proposition yields a sufficient condition for a regular sequence $\xx=x_1,\dots,x_n$ on an $R$-module $M$ to be such that $M$ is a direct summand of $\Omega^n(M/\xx M)$.
Thanks to this proposition, we will be able to prove the main result of this section.

\begin{prop}\label{omegan}
Let $M$ be an $R$-module.
Let $\xx=x_1,\dots,x_n$ be an $R$- and $M$-sequence in $\bigcap_{1\le i,j\le n}\Ann\Ext^i(M,\Omega^jM)$.
Then one has an $R$-isomorphism
$$
\Omega_R^n(M/\xx M)\cong\bigoplus_{i=0}^n(\Omega^iM)^{\oplus\binom{n}{i}}.
$$
\end{prop}

\begin{proof}
We prove the proposition by induction on $n$.
The assertion is trivial when $n=0$.
Let $n\ge 1$, and set $\yy=x_1,\dots,x_{n-1}$.
The induction hypothesis implies that $N:=\Omega^{n-1}(M/\yy M)$ is isomorphic to $\bigoplus_{i=0}^{n-1}(\Omega^iM)^{\oplus\binom{n-1}{i}}$.
We have
$$
\Ann\Ext^1(N,\Omega N)=\Ann\Ext^1(\bigoplus_{i=0}^{n-1}\Omega^iM,\bigoplus_{j=0}^{n-1}\Omega^{j+1}M)=\bigcap_{1\le i,j\le n}\Ann\Ext^i(M,\Omega^jM),
$$
which contains the element $x_n$.
When $n=1$, we have $N=M$ and $x_n=x_1$, hence $x_n$ is $N$-regular.
When $n\ge 2$, we have that $N$ is a submodule of a free $R$-module, hence $x_n$ is $N$-regular since it is $R$-regular.
Therefore in any case $x_n$ is an $N$-regular element.
Using Lemma \ref{syznzd}, we obtain isomorphisms
\begin{align*}
\Omega(N/x_nN) & \cong \Omega N\oplus N \cong (\bigoplus_{i=0}^{n-1}(\Omega^{i+1}M)^{\oplus\binom{n-1}{i}})\oplus(\bigoplus_{i=0}^{n-1}(\Omega^iM)^{\oplus\binom{n-1}{i}}) \\
& \cong \Omega^nM\oplus(\bigoplus_{i=1}^{n-1}(\Omega^iM)^{\oplus(\binom{n-1}{i-1}+\binom{n-1}{i})})\oplus M \cong \bigoplus_{i=0}^n(\Omega^iM)^{\oplus\binom{n}{i}}.
\end{align*}
Thus it is enough to prove that $\Omega^n(M/\xx M)$ is isomorphic to $\Omega(N/x_nN)$.
There is a commutative diagram
$$
\begin{CD}
@. 0 @. 0 @. @. 0 @. 0 \\
@. @VVV @VVV @. @VVV @VVV \\
0 @>>> N @>>> F_{n-2} @. \to \cdots \to @. F_0 @>>> M/\yy M @>>> 0 \\
@. @V{x_n}VV @V{x_n}VV @. @V{x_n}VV @V{x_n}VV \\
0 @>>> N @>>> F_{n-2} @. \to \cdots \to @. F_0 @>>> M/\yy M @>>> 0 \\
@. @VVV @VVV @. @VVV @VVV \\
0 @>>> N/x_nN @>>> F_{n-2}/x_nF_{n-2} @. \to \cdots \to @. F_0/x_nF_0 @>>> M/\xx M @>>> 0 \\
@. @VVV @VVV @. @VVV @VVV \\
@. 0 @. 0 @. @. 0 @. 0
\end{CD}
$$
of $R$-modules with exact rows and columns, where each $F_i$ is free.
Applying $\Omega$ to the third row, we get an exact sequence
$$
0 \to \Omega(N/x_nN) \to F_{n-2}\oplus G_{n-2} \to \cdots \to F_0\oplus G_0 \to \Omega(M/\xx M) \to 0
$$
of $R$-modules, where each $G_i$ is free.
We obtain isomorphisms $\Omega(N/x_nN)\cong\Omega^{n-1}(\Omega(M/\xx M))\oplus H\cong\Omega^n(M/\xx M)\oplus H$ for some free $R$-module $H$.
Since there are exact sequences
\begin{align*}
& 0 \to \Tor_1^R(N,k) \to \Tor_1^R(N/x_nN,k) \to N\otimes_Rk \to 0\quad\text{and}\\
& 0 \to \Tor_n^R(M/\yy M,k) \to \Tor_n^R(M/\xx M,k) \to \Tor_{n-1}^R(M/\yy M,k) \to 0,
\end{align*}
we have the following equalities of Betti numbers:
$$
\beta_1^R(N/x_nN)=\beta_1^R(N)+\beta_0^R(N)=\beta_n^R(M/\yy M)+\beta_{n-1}^R(M/\yy M)=\beta_n^R(M/\xx M).
$$
It follows that $\Omega(N/x_nN)$ and $\Omega^n(M/\xx M)$ have the same minimal number of generators, which implies $H=0$.
Consequently, $\Omega(N/x_nN)$ is isomorphic to $\Omega^n(M/\xx M)$, which we have wanted to prove.
\end{proof}

The lemma below is easily shown by induction on the length of the module $M$.

\begin{lem}\label{length}
Let $M$ be an $R$-module of finite length.
Then $\Omega^nM$ belongs to $\ext(\Omega^nk)$ for every $n\ge 0$.
\end{lem}

Now, we can prove the following theorem concerning the structure of modules which are free on the punctured spectrum.
It is the main result of this section, which will play an essential role in the proofs of the other main results of this paper.

\begin{thm}\label{extdepth}
Let $R$ be a Cohen-Macaulay local ring.
Let $M$ be an $R$-module of depth $t$.
Suppose that $M$ is free on the punctured spectrum of $R$.
Then $M$ belongs to $\ext(\bigoplus_{i=t}^d\Omega^ik)$.
\end{thm}

\begin{proof}
We begin with stating and proving the following claim.

\setcounter{claim}{0}

\begin{claim}\label{rm}
There exists an $R$- and $M$-sequence $\xx=x_1,\dots,x_t$ such that $M$ is isomorphic to a direct summand of $\Omega_R^t(M/\xx M)$.
\end{claim}

\begin{cpf}
There is nothing to prove when $t=0$, so let $t\ge 1$.

First, we consider the case where $M$ is a free $R$-module.
In this case, we have $t=d$.
Take an $R$-sequence $\xx=x_1,\dots,x_d$.
Then $\xx$ is also an $M$-sequence.
The Koszul complex of $\xx$ with respect to $M$ yields an exact sequence
$$
0 \to M^{\oplus\binom{d}{d}} \to M^{\oplus\binom{d}{d-1}} \to \cdots \to M^{\oplus\binom{d}{1}} \to M^{\oplus\binom{d}{0}} \to M/\xx M \to 0
$$
of $R$-modules.
Note that this gives a minimal free resolution of the $R$-module $M/\xx M$.
We see from this exact sequence that $M=M^{\oplus\binom{d}{d}}$ is isomorphic to $\Omega_R^d(M/\xx M)$.

Next, let us consider the case where $M$ is not a free $R$-module.
Put $I=\bigcap_{1\le i,j\le t}\Ann\Ext^i(M,\Omega^jM)$.
Since $M$ is free on the punctured spectrum, $\Ann\Ext^i(M,\Omega^jM)$ is either a unit ideal or an $\m$-primary ideal of $R$ for $1\le i,j\le t$.
The ideal $I$ is contained in $\Ann\Ext^1(M,\Omega M)$, which is not a unit ideal by Proposition \ref{nfsupp}(2).
Hence $I$ is an $\m$-primary ideal of $R$.
Since $\depth R=d\ge t=\depth M$, we can choose an $R$- and $M$-sequence $\xx=x_1,\dots,x_t$ in $I$.
Proposition \ref{omegan} implies that $M$ is isomorphic to a direct summand of $\Omega^t(M/\xx M)$.
Thus, the proof of the claim is completed.
\qed
\end{cpf}

Let us prove the theorem by induction on $d-t$.
When $d-t=0$, the sequence $\xx$ is a system of parameters of $R$, hence the $R$-module $M/\xx M$ has finite length.
By Lemma \ref{length}, $\Omega^t(M/\xx M)$ belongs to $\ext(\Omega^tk)=\ext(\bigoplus_{i=t}^d\Omega^ik)$, and so does $M$.
Let $d-t>0$.
There is an exact sequence $0 \to L \to M/\xx M \to N \to 0$ of $R$-modules such that $L$ has finite length and that $N$ has positive depth.
From this we get an exact sequence $0 \to \Omega^tL \to \Omega^t(M/\xx M)\oplus F \to \Omega^tN \to 0$ with $F$ free.
Lemma \ref{length} says that $\Omega^tL$ is in $\ext(\Omega^tk)$, and hence it is in $\ext(\bigoplus_{i=t}^d\Omega^ik)$.
We have only to show that $\Omega^tN$ is in $\ext(\bigoplus_{i=t}^d\Omega^ik)$.
Here we claim the following.

\begin{claim}\label{fops}
The $R$-module $\Omega^tN$ is free on the punctured spectrum of $R$.
\end{claim}

\begin{cpf}
Fix a nonmaximal prime ideal $\p$ of $R$.
As $L_\p=0$, the module $N_\p$ is isomorphic to $M_\p/\xx M_\p$.
If the sequence $\xx$ is not in $\p$ or if $\p$ is not in $\Supp M$, then $N_\p=0$, and $(\Omega_R^tN)_\p$ is isomorphic to $\Omega_{R_\p}^tN_\p=0$ up to free summand.
Hence $(\Omega_R^tN)_\p$ is $R_\p$-free.
If $\xx$ is in $\p$ and $\p$ is in $\Supp M$, then $\xx$ forms an $M_\p$-sequence.
The $R_\p$-module $(\Omega_R^tN)_\p$ is isomorphic to $\Omega_{R_\p}^tN_\p$ up to free summand, and we have isomorphisms $\Omega_{R_\p}^tN_\p\cong\Omega_{R_\p}^t(M_\p/\xx M_\p)\cong M_\p$ similarly to the argument at the beginning of the proof of the theorem.
Since $M_\p$ is $R_\p$-free, so is $\Omega_{R_\p}^tN_\p$, and so is $(\Omega_R^tN)_\p$.
\qed
\end{cpf}

We have $s:=\depth\Omega^tN=\min\{\depth N+t,d\}>t$, hence $d-s<d-t$.
By Claim \ref{fops} we can apply the induction hypothesis to $\Omega^tN$, and see that $\Omega^tN$ belongs to $\ext(\bigoplus_{i=s}^d\Omega^ik)$, which is contained in $\ext(\bigoplus_{i=t}^d\Omega^ik)$.
This completes the proof of the theorem.
\end{proof}

\begin{rem}
In Theorem \ref{extdepth}, the module $M$ does not necessarily belong to $\ext(\Omega^tk)$.
In fact, let $R$ be a Cohen-Macaulay local ring of positive dimension.
Take an extension $0 \to k \to M \to \m \to 0$ of $\m$ by $k$.
Then the middle term $M$ is an $R$-module of depth $0$ which is free on the punctured spectrum of $R$.
Note that the category $\ext(\Omega^0k)=\ext(k)$ consists of all $R$-modules of finite length.
Since $M$ does not have finite length, $M$ does not belong to $\ext(\Omega^0k)$.
\end{rem}

The result below is a direct consequence of Theorem \ref{extdepth}.

\begin{cor}\label{kore}
Let $R$ be a Cohen-Macaulay local ring.
Let $M$ be a Cohen-Macaulay $R$-module which is free on the punctured spectrum of $R$.
Then $M$ belongs to $\ext(\Omega^dk)$.
\end{cor}

Recall that $R$ is said to have an {\em isolated singularity} if $R_\p$ is a regular local ring for every nonmaximal prime ideal $\p$ of $R$.
The following result is immediately obtained from Corollary \ref{kore}.

\begin{cor}\label{ceodk}
Let $R$ be a Cohen-Macaulay local ring having an isolated singularity.
Then one has $\CM(R)=\ext(\Omega^dk)$.
\end{cor}

\begin{rem}
The conclusion of Corollary \ref{ceodk} does not mean that every Cohen-Macaulay $R$-module can be obtained by taking extensions finitely many times from the indecomposable summands of $\Omega^dk$.
So, the conclusion of the corollary does not mean that the Grothendieck group of $\CM(R)$ is finitely generated.
\end{rem}

Using Corollary \ref{ceodk}, we can get the following structure result of the stable category of Cohen-Macaulay modules.

\begin{cor}\label{omegadk}
Let $R$ be a Gorenstein local ring with an isolated singularity.
Then the thick subcategory of the triangulated category $\lCM(R)$ generated by $\Omega^dk$ coincides with $\lCM(R)$.
\end{cor}

\begin{proof}
Let $\Y$ be the thick subcategory of $\lCM(R)$ generated by $\Omega^dk$.
Fix a Cohen-Macaulay $R$-module $M$.
We want to prove that $M$ belongs to $\Y$.
Since $M$ is in $\ext(\Omega^dk)$ by Corollary \ref{ceodk}, $M$ is in $\ext^n(\Omega^dk)$ for some $n\ge 0$.
Let us show that $M$ belongs to $\Y$ by induction on $n$.
When $n=0$, the module $M$ is in $\add(\Omega^dk)$, and clearly $M$ belongs to $\Y$.
Let $n\ge 1$.
There are exact sequences
$$
0 \to L_i \overset{f_i}{\to} M_i \overset{g_i}{\to} N_i \to 0\quad (1\le i\le s)
$$
of $R$-modules with $L_i,N_i\in\ext^{n-1}(\Omega^dk)$ such that $M$ is a direct summand of the direct sum $M_1\oplus\cdots\oplus M_s$.
Note that $L_i,N_i$ are Cohen-Macaulay $R$-modules.
The induction hypothesis implies that $L_i,N_i$ belong to $\Y$.
For each $1\le i\le s$, there is an exact sequence $0 \to L_i \to F_i \to \Sigma L_i \to 0$ in $\CM(R)$ such that $F_i$ is free.
We make the following pushout diagram.
$$
\begin{CD}
@. 0 @. 0 \\
@. @VVV @VVV \\
0 @>>> L_i @>>> F_i @>>> \Sigma L_i @>>> 0 \\
@. @V{f_i}VV @VVV @| \\
0 @>>> M_i @>>> N_i' @>>> \Sigma L_i @>>> 0 \\
@. @V{g_i}VV @VVV \\
@. N_i @= N_i \\
@. @VVV @VVV \\
@. 0 @. 0
\end{CD}
$$
Since $R$ is Gorenstein and $N_i$ is Cohen-Macaulay over $R$, the middle column splits.
Hence we get an exact triangle $L_i \overset{f_i}{\to} M_i \overset{g_i}{\to} N_i \to \Sigma L_i$ in $\lCM(R)$.
The thickness of $\Y$ shows that $M_i$ belongs to $\Y$ for each $1\le i\le s$, which implies that $M$ also belongs to $\Y$.
\end{proof}

\begin{rem}
\begin{enumerate}[(1)]
\item
Corollary \ref{omegadk} can also be proved by using \cite[Theorem VI.8]{Scho}.
\item
In Corollary \ref{omegadk}, the assumption that $R$ has an isolated singularity is indispensable.
We will actually observe in Remark \ref{yogen} that the assertion of Corollary \ref{omegadk} does not necessarily hold without that assumption.
\item
A generalization of Corollary \ref{omegadk} will be obtained in Theorem \ref{main}(2).
\end{enumerate}
\end{rem}

%%%%%%%%%%%%%%%%%%%%%%%%%%%%%%%%%%%%%%%%%%%%%%%%%%%%%%%%%%%%%
%%%%%%%%%%%%%%%%%%%%%%%%%%%%%%%%%%%%%%%%%%%%%%%%%%%%%%%%%%%%%
\section{Cohen-Macaulay modules and completion}\label{cmac}
%%%%%%%%%%%%%%%%%%%%%%%%%%%%%%%%%%%%%%%%%%%%%%%%%%%%%%%%%%%%%
%%%%%%%%%%%%%%%%%%%%%%%%%%%%%%%%%%%%%%%%%%%%%%%%%%%%%%%%%%%%%

In this section, we compare Cohen-Macaulay $R$-modules and Cohen-Macaulay $\widehat R$-modules, where $\widehat R$ denotes the $\m$-adic completion of $R$.
We do this by using the structure result of modules that are free on the punctured spectrum, which was obtained in the previous section.
We start by investigating the relationship between extension closures and completion.

\begin{prop}\label{exthat}
Let $n$ be a nonnegative integer, and let $X$ be an $R$-module which is free on the punctured spectrum of $R$.
Then for any $\widehat R$-module $N$ in $\ext_{\widehat R}^n\widehat X$ there exists an $R$-module $M$ in $\ext_R^nX$ such that $N$ is isomorphic to a direct summand of $\widehat M$.
\end{prop}

\begin{proof}
We prove the proposition by induction on $n$.
When $n=0$, the $\widehat R$-module $N$ is in $\add_{\widehat R}\widehat X$, and $N$ is isomorphic to a direct summand of $(\widehat X)^{\oplus m}$ for some $m\ge 0$.
Then we can take $M:=X^{\oplus m}\in\add_RX$.
When $n\ge 1$, there is an exact sequence $0 \to A \overset{f}{\to} B \overset{g}{\to} C \to 0$ of $\widehat R$-modules with $A,C\in\ext_{\widehat R}^{n-1}(\widehat X)$ such that $N$ is a direct summand of $B$.
The induction hypothesis implies that there exist $S,T\in\ext_R^{n-1}X$ such that $A$ and $C$ are isomorphic to direct summands of $\widehat S$ and $\widehat T$, respectively.
Hence we have isomorphisms $A\oplus A'\cong\widehat S$ and $C\oplus C'\cong\widehat T$ of $\widehat R$-modules.
Taking the direct sum of the identity map $A'\to A'$ and $f$, and the direct sum of the identity map $C'\to C'$ and $g$, we obtain an exact sequence $\sigma:0 \to \widehat S \to B' \to \widehat T \to 0$ of $\widehat R$-modules such that $N$ is a direct summand of $B'$.
The exact sequence $\sigma$ can be regarded as an element of $\Ext_{\widehat R}^1(\widehat T,\widehat S)$, which is isomorphic to the completion $(\Ext_R^1(T,S))^{\widehat\quad}$ of $\Ext_R^1(T,S)$.
By using Proposition \ref{clloc}(2) and the assumption that $X$ is free on the punctured spectrum of $R$, we observe that $T$ is free on the punctured spectrum of $R$.
Hence the $R$-module $\Ext_R^1(T,S)$ has finite length, and so we have $(\Ext_R^1(T,S))^{\widehat\quad}\cong\Ext_R^1(T,S)$.
Therefore there exists an exact sequence $\tau: 0 \to S \overset{\alpha}{\to} M \overset{\beta}{\to} T \to 0$ of $R$-modules such that $\sigma$ is equivalent to the exact sequence $\widehat\tau: 0 \to \widehat S \overset{\widehat\alpha}{\to} \widehat M \overset{\widehat\beta}{\to} \widehat T \to 0$.
By the choice of $S$ and $T$, the module $M$ belongs to $\ext_R^nX$, and $B'$ is isomorphic to $\widehat M$.
This finishes the proof of the proposition.
\end{proof}

If $R$ is Cohen-Macaulay, then every $\widehat R$-module that is free on the punctured spectrum of $\widehat R$ is, up to direct summand, the completion of some $R$-module that is free on the punctured spectrum of $R$:

\begin{thm}\label{dephat}
Let $R$ be a Cohen-Macaulay local ring.
Then for any $\widehat R$-module $N$ which is free on the punctured spectrum of $\widehat R$, there exists an $R$-module $M$ which is free on the punctured spectrum of $R$ with $\depth_RM=\depth_{\widehat R}N$ such that $N$ is isomorphic to a direct summand of $\widehat M$.
\end{thm}

This result is already known: Wiegand \cite{W} essentially proves a more general result by applying Elkik's theorem \cite{E2}.
It is explicitly stated in \cite[Corollary 3.5]{FSWW}.
We give here a very different proof.

\begin{proof}
By Theorem \ref{extdepth}, $N$ belongs to $\ext_{\widehat R}(\bigoplus_{i=t}^d\Omega_{\widehat R}^ik)=\ext_{\widehat R}(\widehat{L})$, where $t=\depth_{\widehat R}N$ and $L=\bigoplus_{i=t}^d\Omega_R^ik$.
Since $L$ is free on the punctured spectrum of $R$, it follows from Proposition \ref{exthat} that there exists an $R$-module $M\in\ext_R(L)$ such that the $\widehat R$-module $N$ is isomorphic to a direct summand of $\widehat M$.
We see from Proposition \ref{clloc}(2) that $M$ is free on the punctured spectrum of $R$.
We have $\depth_{\widehat R}N\ge\depth_{\widehat R}\widehat M=\depth_RM$, and $\depth_RM\ge\depth_RL=t=\depth_{\widehat R}N$ by Proposition \ref{cldep}(1).
Thus the equality $\depth_RM=\depth_{\widehat R}N$ holds.
\end{proof}

The corollary below immediately follows from Theorem \ref{dephat}.

\begin{cor}\label{pfhat}
Let $R$ be a Cohen-Macaulay local ring.
Then for any Cohen-Macaulay $\widehat R$-module $N$ which is free on the punctured spectrum of $\widehat R$, there exists a Cohen-Macaulay $R$-module $M$ which is free on the punctured spectrum of $R$ such that $N$ is isomorphic to a direct summand of $\widehat M$.
\end{cor}

Here we recall a well-known elementary fact.

\begin{prop}\label{kisoteki}
If $\widehat R$ has an isolated singularity, then so does $R$.
The converse of the first assertion holds if $R$ is excellent.
\end{prop}

\begin{rem}
It is known that every complete local ring containing a field is the completion of some local ring having an isolated singularity; see \cite{He}.
Hence, in general, having an isolated singularity does not ascend to the completion.
\end{rem}

The following result is a consequence of Proposition \ref{kisoteki} and Corollary \ref{pfhat}.

\begin{cor}\label{ishat}
Let $R$ be a Cohen-Macaulay local ring whose completion has an isolated singularity (e.g. a Cohen-Macaulay excellent local ring with an isolated singularity).
Then for any Cohen-Macaulay $\widehat R$-module $N$ there exists a Cohen-Macaulay $R$-module $M$ such that $N$ is isomorphic to a direct summand of $\widehat M$.
\end{cor}

We recall a definition concerning an equivalence of categories.

\begin{dfn}
Let $F:\A\to\B$ be an additive functor of additive categories.
We say that $F$ is an {\em equivalence up to direct summand} if it satisfies the following two conditions.
\begin{enumerate}[(1)]
\item
$F$ is fully faithful.
\item
$F$ is essentially dense, namely, for each object $B$ of $\B$ there exists an object $A$ of $\A$ such that $B$ is isomorphic to a direct summand of $FA$.
\end{enumerate}
\end{dfn}

Now, by using Corollary \ref{ishat}, we recover a recent result of Keller, Murfet and Van den Bergh \cite[Corollary A.7]{KMV}.

\begin{cor}[Keller-Murfet-Van den Bergh]\label{k-vdb}
Let $R$ be a Cohen-Macaulay local ring whose completion has an isolated singularity (e.g. a Cohen-Macaulay excellent local ring with an isolated singularity).
Then the natural functor $\lCM(R)\to\lCM(\widehat R)$ is an equivalence up to direct summand.
\end{cor}

\begin{proof}
Let $M,N\in\lCM(R)$.
We have functorial isomorphisms
\begin{align*}
\lhom_R(M,N) & \cong\Tor_1^R(\tr_RM,N)\cong(\Tor_1^R(\tr_RM,N))^{\widehat\quad}\\
& \cong\Tor_1^{\widehat R}(\tr_{\widehat R}\widehat M,\widehat N)\cong\lhom_{\widehat R}(\widehat M,\widehat N),
\end{align*}
where the second isomorphism follows from the fact that the $R$-module $\Tor_1^R(\tr_RM,N)$ has finite length since $N$ is free on the punctured spectrum of $R$.
Therefore the functor $\lCM(R)\to\lCM(\widehat R)$ is fully faithful.
The essential density of this functor follows from Corollary \ref{ishat}.
\end{proof}

Applying Corollary \ref{pfhat}, we can also obtain some results concerning ascent and descent of the Cohen-Macaulay representation type of Cohen-Macaulay local rings.
For this, we state an easy lemma.

\begin{lem}\label{asdes}
An $R$-module $M$ is free on the punctured spectrum of $R$ if and only if $\widehat M$ is free on the punctured spectrum of $\widehat R$.
\end{lem}

Now we can prove the following proposition.
This result says that finiteness and countability of the set of isomorphism classes of indecomposable Cohen-Macaulay modules that are free on the punctured spectrum ascends and descends between a Cohen-Macaulay local ring and its completion.

\begin{prop}\label{fcfc}
Let $R$ be a Cohen-Macaulay local ring.
Then the following are equivalent:
\begin{enumerate}[\rm (1)]
\item
There exist only finitely (respectively, countably) many isomorphism classes of indecomposable Cohen-Macaulay $R$-modules that are free on the punctured spectrum of $R$;
\item
There exist only finitely (respectively, countably) many isomorphism classes of indecomposable Cohen-Macaulay $\widehat R$-modules that are free on the punctured spectrum of $\widehat R$.
\end{enumerate}
\end{prop}

This result is proved in a similar way to the proof of \cite[Theorem 1.4]{W}.

\begin{proof}
(1) $\Rightarrow$ (2): Let $M_1,M_2,\dots,M_n$ be all the nonisomorphic indecomposable Cohen-Macaulay $R$-modules that are free on the punctured spectrum of $R$.
(Here, we regard $n$ as $\infty$ in the countable case.)
Let $\widehat{M_i} \cong L_{i,1}\oplus\cdots\oplus L_{i,l_i}$ be an indecomposable decomposition of the $\widehat R$-module $\widehat{M_i}$ for each integer $i$ with $1\le i\le n$.
We claim that every indecomposable Cohen-Macaulay $\widehat R$-module $N$ that is free on the punctured spectrum of $\widehat R$ is isomorphic to some $L_{i,j}$.
Indeed, Corollary \ref{pfhat} guarantees that there exists a Cohen-Macaulay $R$-module $M$ which is free on the punctured spectrum of $R$ such that $N$ is isomorphic to a direct summand of $\widehat M$.
Since every (indecomposable) direct summand of $M$ is a Cohen-Macaulay $R$-module that is free on the punctured spectrum of $R$, we have an indecomposable decomposition $M\cong M_1^{\oplus a_1}\oplus\cdots\oplus M_m^{\oplus a_m}$ of the $R$-module $M$ for some integer $m$ with $1\le m\le n$.
Hence $N$ is isomorphic to a direct summand of $\widehat{M_1}^{\oplus a_1}\oplus\cdots\oplus\widehat{M_m}^{\oplus a_m}$.
Note that $\widehat R$ is a henselian local ring.
Since $N$ is an indecomposable $\widehat R$-module, by virtue of the Krull-Schmidt theorem $N$ is isomorphic to a direct summand of $\widehat M_k$ for some $1\le k\le m$, and is isomorphic to $L_{k,h}$ for some $1\le h\le l_k$.

(2) $\Rightarrow$ (1): Let $N_1,N_2,\dots,N_n$ be all the nonisomorphic indecomposable Cohen-Macaulay $\widehat R$-modules that are free on the punctured spectrum of $\widehat R$.
(Here, we regard $n$ as $\infty$ in the countable case.)
Corollary \ref{pfhat} shows that for each integer $i$ with $1\le i\le n$ there exists a Cohen-Macaulay $R$-module $M_i$ which is free on the punctured spectrum of $R$ such that $N_i$ is isomorphic to a direct summand of $\widehat{M_i}$.
Put $M_i'=M_1\oplus\cdots\oplus M_i$.
It follows from \cite[Theorem 1.1]{W} that there are only finitely many nonisomorphic indecomposable $R$-modules in $\add_R(M_i')$.
Let $L_{i,1},\dots,L_{i,l_i}$ be those $R$-modules.
Then we claim that every indecomposable Cohen-Macaulay $R$-module $M$ that is free on the punctured spectrum of $R$ is isomorphic to some $L_{i,j}$.
In fact, Lemma \ref{asdes} implies that $\widehat M$ is free on the punctured spectrum of $\widehat R$.
Hence for some integer $m$ with $1\le m\le n$ there is an isomorphism $\widehat M\cong N_1^{\oplus a_1}\oplus\cdots\oplus N_m^{\oplus a_m}$ of $\widehat R$-modules, which implies that $\widehat M$ belongs to $\add_{\widehat R}(\widehat{M_m'})$.
According to \cite[Lemma 1.2]{W}, the $R$-module $M$ belongs to $\add_R(M_m')$, and it follows that $M$ is isomorphic to $L_{m,k}$ for some $1\le k\le l_m$.
\end{proof}

A Cohen-Macaulay local ring $R$ is said to have {\em finite} (respectively, {\em countable}) {\em Cohen-Macaulay representation type} if there exist only finitely (respectively, countably) many isomorphism classes of indecomposable Cohen-Macaulay $R$-modules.
Now we recover the following theorem due to Leuschke and Wiegand \cite[Corollary 1.6]{W}\cite[Main Theorem]{LW}, which was conjectured by Schreyer \cite[Conjecture 7.3]{S}.
In fact, Propositions \ref{fcfc}, \ref{kisoteki} and \cite[Corollary 2]{HL2} imply this corollary.

\begin{cor}\label{wlw}
Let $R$ be a Cohen-Macaulay local ring.
\begin{enumerate}[\rm (1)]
\item
If the completion $\widehat R$ has finite Cohen-Macaulay representation type, then so does $R$.
\item
The converse of the first assertion holds if either $\widehat R$ has an isolated singularity or $R$ is excellent.
\end{enumerate}
\end{cor}

Using Propositions \ref{fcfc} and \ref{kisoteki}, we can also obtain the following result.

\begin{cor}\label{wlwc}
Let $R$ be a Cohen-Macaulay local ring whose completion $\widehat R$ has an isolated singularity (e.g. a Cohen-Macaulay excellent local ring with an isolated singularity).
Then $R$ has countable Cohen-Macaulay representation type if and only if so does $\widehat R$.
\end{cor}

%%%%%%%%%%%%%%%%%%%%%%%%%%%%%%%%%%%%%%%%%%%%%%%%%%%%%%%%%%%%%
%%%%%%%%%%%%%%%%%%%%%%%%%%%%%%%%%%%%%%%%%%%%%%%%%%%%%%%%%%%%%
\section{Thick subcategories of Cohen-Macaulay modules}\label{tsocm}
%%%%%%%%%%%%%%%%%%%%%%%%%%%%%%%%%%%%%%%%%%%%%%%%%%%%%%%%%%%%%
%%%%%%%%%%%%%%%%%%%%%%%%%%%%%%%%%%%%%%%%%%%%%%%%%%%%%%%%%%%%%

In this section, we consider classifying thick subcategories of $\CM(R)$ in terms of specialization-closed subsets of $\Spec R$.
We begin with a lemma.

\begin{lem}\label{zurasi}
Let $0 \to L \to M \to N \to 0$ be an exact sequence of $R$-modules.
Then there exists an exact sequence $0 \to \Omega N \to L\oplus F \to M \to 0$ of $R$-modules, where $F$ is free.
\end{lem}

This lemma is easily proved by taking an exact sequence $0 \to \Omega N \to F \to N \to 0$ with $F$ free and making a pullback diagram.

The following result is easy to verify by using the fact that $\CM(R)$ is a resolving subcategory of $\mod R$ (cf. Example \ref{resex}(3)).

\begin{prop}\label{thres1}
Let $R$ be a Cohen-Macaulay local ring.
If $\X$ is a thick subcategory of $\CM(R)$ containing $R$, then $\X$ is a resolving subcategory of $\mod R$ contained in $\CM(R)$
\end{prop}

For a subcategory $\X$ of $\mod R$, we denote by $\widetilde\X$ the subcategory of $\mod R$ consisting of all modules $M$ such that there exists an exact sequence
$$
0 \to X_n \to X_{n-1} \to \cdots \to X_1 \to X_0 \to M \to 0
$$
with $X_i\in\X$ for $0\le i\le n$.
To investigate the structure of $\widetilde\X$ for a thick subcategory $\X$ of $\CM(R)$, we make the following lemma.

\begin{lem}\label{resome}
Let $\X$ be a resolving subcategory of $\mod R$.
Let $0 \to X_n \to X_{n-1} \to \cdots \to X_1 \to X_0 \to M \to 0$ be an exact sequence in $\mod R$ with $X_i\in\X$ for $0\le i\le n$.
Then $\Omega^nM$ belongs to $\X$.
\end{lem}

\begin{proof}
Let us prove the lemma by induction on $n$.
The assertion obviously holds in the case where $n=0$.
Let $n\ge 1$.
Let $N$ be the kernel of the map $X_0\to M$ in the exact sequence.
Then there is an exact sequence $0 \to X_n \to X_{n-1} \to \cdots \to X_2 \to X_1 \to N \to 0$, and the induction hypothesis says that $\Omega^{n-1}N$ belongs to $\X$.
Applying Lemma \ref{zurasi} to the exact sequence $0 \to N \to X_0 \to M \to 0$, we get an exact sequence $0 \to \Omega M \to N\oplus F \to X_0 \to 0$, where $F$ is a free $R$-module.
From this we obtain an exact sequence $0 \to \Omega^nM \to \Omega^{n-1} N\oplus G \to \Omega^{n-1}X_0 \to 0$ of $R$-modules such that $G$ is free.
Since $\Omega^{n-1}X_0$ and $\Omega^{n-1}M\oplus G$ belong to $\X$, so does $\Omega^nM$, as required.
\end{proof}

Here we recall the definition of a Cohen-Macaulay approximation.

\begin{dfn}
Let $R$ be a Cohen-Macaulay local ring with a canonical module.
Let $M$ be an $R$-module.
Let $0 \to I \to C \to M \to 0$ be an exact sequence of $R$-modules such that $C$ is Cohen-Macaulay and that $I$ has finite injective dimension.
Such a Cohen-Macaulay module $C$ is called a {\em Cohen-Macaulay approximation} of $M$.
For every $R$-module $M$, a Cohen-Macaulay approximation of $M$ exists.
It is not uniquely determined in general, but if $R$ is henselian, then it is essentially uniquely determined.
For the details of the notion of a Cohen-Macaulay approximation, see \cite{ABu}.
\end{dfn}

Let $\X$ be a thick subcategory of $\CM(R)$ containing $R$ and the canonical module of $R$.
Then we have several equivalent conditions for a given $R$-module to be in the subcategory $\widetilde\X$ of $\mod R$.

\begin{prop}\label{7}
Let $R$ be a Cohen-Macaulay local ring with a canonical module $\omega$.
Let $\X$ be a thick subcategory of $\CM(R)$ containing $R$ and $\omega$.
Put $n=d-\depth M$.
Then the following are equivalent:
\begin{enumerate}[\rm (1)]
\item
$M$ is in $\widetilde\X$;
\item
$\Omega^iM$ is in $\X$ for some $i\ge 0$;
\item
$\Omega^dM$ is in $\X$;
\item
$\Omega^nM$ is in $\X$;
\item
$\Omega^iM$ is in $\X$ for all $i\ge n$;
\item
Every Cohen-Macaulay approximation of $M$ is in $\X$;
\item
Some Cohen-Macaulay approximation of $M$ is in $\X$.
%\item
%If $0 \to N \to X_{n-1} \to \cdots \to X_0 \to M \to 0$ is an exact sequence of $R$-modules with $X_i\in\X$ for $0\le i\le {n-1}$, then $N$ is in $\X$;
%\item
%Every Cohen-Macaulay approximation of $M$ is in $\X$;
\end{enumerate}
\end{prop}

\begin{proof}
First of all, note from Proposition \ref{thres1} that the subcategory $\X$ of $\mod R$ is resolving, and so it is closed under syzygies.

(5) $\Leftrightarrow$ (4): These implications are trivial.

(4) $\Rightarrow$ (3): The module $\Omega^dM$ is the $(d-n)^\text{th}$ syzygy of $\Omega^nM$.
Hence $\Omega^dM$ belongs to $\X$.

(3) $\Rightarrow$ (2): This implication is trivial.

(2) $\Rightarrow$ (4): Since $\X$ is a subcategory of $\CM(R)$, the $R$-module $\Omega^iM$ is Cohen-Macaulay.
Hence we have $d=\depth\Omega^iM=\inf\{\depth M+i,d\}$, which implies $\depth M+i\ge d$, i.e., $i\ge n$.
Therefore there is an exact sequence
$$
0 \to \Omega^iM \overset{f_i}{\to} F_{i-1} \overset{f_{i-1}}{\to} \cdots \overset{f_{n+2}}{\to} F_{n+1} \overset{f_{n+1}}{\to} F_n \overset{f_n}{\to} \Omega^nM \to 0
$$
of $R$-modules such that each $F_j$ is free.
Letting $N_j$ be the image of $f_j$, we have a short exact sequence $0 \to N_{j+1} \to F_j \to N_j \to 0$ for each $n\le j\le i-1$.
Note that all modules appearing in this short exact sequence are Cohen-Macaulay.
Since $\X$ is a thick subcategory of $\CM(R)$, we see by descending induction on $j$ that $\Omega^nM$ belongs to $\X$.

(1) $\Rightarrow$ (2): Lemma \ref{resome} shows this implication.

(2) $\Rightarrow$ (1): There is an exact sequence $0 \to \Omega^iM \to F_{i-1} \to \cdots \to F_1 \to F_0 \to M \to 0$ such that $F_0,F_1,\dots,F_{i-1}$ are free.
Since $F_0,F_1,\dots,F_{i-1},\Omega^iM$ are all in $\X$, the module $M$ is in $\widetilde\X$.

The combination of all the above arguments proves that the conditions (1)--(5) are equivalent to one another.

(6) $\Rightarrow$ (7): This implication is clear.

(7) $\Rightarrow$ (1) and (3) $\Rightarrow$ (6): Let
\begin{equation}\label{icm}
0 \to I \to C \to M \to 0
\end{equation}
be an exact sequence of $R$-modules such that $C$ is Cohen-Macaulay and that $I$ has finite injective dimension.
According to \cite[Exercise 3.3.28(b)]{BH}, there is an exact sequence
\begin{equation}\label{oi}
0 \to \omega_m \to \cdots \to \omega_1 \to \omega_0 \to I \to 0
\end{equation}
with $\omega_j\in\add\omega$ for $0\le j\le m$.
Splicing \eqref{icm} and \eqref{oi} together, we obtain an exact sequence $0 \to \omega_m \to \cdots \to \omega_1 \to \omega_0 \to C \to M \to 0$.
Each $\omega_j$ is in $\X$ since $\X$ contains $\omega$.
Hence, if $C$ belongs to $\X$, then $M$ belongs to $\widetilde\X$.
Thus (7) implies (1).

Now suppose that the condition (3) holds.
From \eqref{icm} we get an exact sequence $0 \to \Omega^d I \to \Omega^d C\oplus F \to \Omega^d M \to 0$ of Cohen-Macaulay $R$-modules with $F$ free.
It follows from the exact sequence \eqref{oi} that $I$ is in $\widetilde\X$.
By the equivalence (1) $\Leftrightarrow$ (3), the module $\Omega^dI$ is in $\X$.
As $\Omega^dM$ is in $\X$, the module $\Omega^dC$ is also in $\X$.
There is an exact sequence $0 \to \Omega^dC \to P_{d-1} \to \cdots \to P_1 \to P_0 \to C \to 0$ where each $P_j$ is free.
Note that all the modules appearing in this exact sequence are Cohen-Macaulay $R$-modules.
Similarly to the proof of the implication (2) $\Rightarrow$ (4), decomposing the above exact sequence into short exact sequences shows that the module $C$ belongs to $\X$.
Thus (3) implies (6).

As a consequence, the conditions (1)--(7) are equivalent to one another.
\end{proof}

To consider classifying thick subcategories of Cohen-Macaulay modules, it is necessary to investigate the structure of the additive closure of a localization of a resolving subcategory of $\mod R$.

\begin{lem}\label{nxm}
Let $\X$ be a resolving subcategory of $\mod R$, and let $\p$ be a prime ideal of $R$.
Suppose that an $R$-module $M$ is such that $M_\p$ belongs to $\add_{R_\p}\X_\p$.
Then there exists an exact sequence $0 \to N \to X \to M \to 0$ in $\mod R$ with $X\in\X$ such that the localized exact sequence $0 \to N_\p \to X_\p \to M_\p \to 0$ splits.
\end{lem}

\begin{proof}
There is an $R_\p$-isomorphism $M_\p\oplus L\cong Y_\p$ for some $L\in\mod R_\p$ and $Y\in\X$.
We have an isomorphism $L\cong K_\p$ for some $K\in\mod R$, and get $(M\oplus K)_\p\cong Y_\p$.
Hence there exists an isomorphism $Y_\p\to (M\oplus K)_\p$ of $R_\p$-modules.
Since $\Hom_{R_\p}(Y_\p,(M\oplus K)_\p)\cong\Hom_R(Y,M\oplus K)_\p$, there is a homomorphism $f:Y\to M\oplus K$ of $R$-modules such that $f_\p$ is an isomorphism.
Write $f=\binom{g}{h}$ for some $g\in\Hom_R(Y,M)$ and $h\in\Hom_R(Y,K)$.
Then $\binom{g_\p}{h_\p}$ is an isomorphism, and it is easy to see that $g_\p$ is a split epimorphism.
There exists an $R$-homomorphism $\rho:F\to M$ with $F$ free such that the $R$-homomorphism $(g,\rho):Y\oplus F\to M$ is surjective.
Taking its kernel and setting $X=Y\oplus F$, we obtain an exact sequence $0 \to N \to X \overset{(g,\rho)}{\longrightarrow} M \to 0$ of $R$-modules with $X\in\X$.
It is easily seen that the localization of this exact sequence at $\p$ splits.
Thus the proof of the lemma is completed.
\end{proof}

Applying the above lemma, we can prove the following proposition.
It will play a key role in the proof of (the essential part of) the main result of this section.

\begin{prop}\label{lnx}
Let $R$ be a Cohen-Macaulay local ring.
Let $\X$ be a resolving subcategory of $\mod R$ contained in $\CM(R)$, and let $M$ be a Cohen-Macaulay $R$-module.
Let $\Phi$ be a nonempty finite subset of $\Spec R$.
Assume $M_\p$ is in $\add_{R_\p}\X_\p$ for every $\p\in\Phi$.
Then there exists an exact sequence $0 \to L \to N \to X \to 0$ of Cohen-Macaulay $R$-modules such that $X\in\X$, that $M$ is a direct summand of $N$, that $\V_R(L)\subseteq\V_R(M)$ and that $\V_R(L)\cap\Phi=\emptyset$.
\end{prop}

\begin{proof}
Let $\Phi=\{\p_1,\dots,\p_n\}$.
Lemma \ref{nxm} shows that for each $1\le i\le n$ there exists an exact sequence $0 \to K_i \to X_i \overset{\phi_i}{\to} M \to 0$ with $X_i\in\X$ such that the localized exact sequence $0 \to (K_i)_{\p_i} \to (X_i)_{\p_i} \overset{(\phi_i)_{\p_i}}{\to} M_{\p_i} \to 0$ splits.
We get an exact sequence $\sigma: 0 \to K \overset{\psi}{\to} X \overset{\phi}{\to} M \to 0$, where $X=X_1\oplus\cdots\oplus X_n\in\X$ and $\phi=(\phi_1,\dots,\phi_n)$.
Note that $X$ and $K$ are Cohen-Macaulay since so are $M$ and each $X_i$.
It is easily seen that the localization of $\sigma$ at $\p_i$ splits for each $1\le i\le n$.
We identify the exact sequence $\sigma$ as the corresponding element of $\Ext_R^1(M,K)$.
Then for each $1\le i\le n$ we have $\sigma_{\p_i}=0$ in $\Ext_R^1(M,K)_{\p_i}$, which implies that there exists an element $f_i\in R\setminus\p_i$ such that $f_i\sigma=0$.
Setting $f=f_1f_2\cdots f_n$, we have $f\notin\p_i$ for all $1\le i\le n$ and $f\sigma=0$.
As $\sigma_f=0$ in $\Ext_R^1(M,K)_f\cong\Ext_{R_f}^1(M_f,K_f)$, the exact sequence $\sigma_f:0 \to K_f \overset{\psi_f}{\to} X_f \overset{\phi_f}{\to} M_f \to 0$ of $R_f$-modules splits.
Since $\Hom_{R_f}(M_f,X_f)$ is isomorphic to $\Hom_R(M,X)_f$, we see that there exists an $R$-homomorphism $\nu:M\to X$ such that $(\nu_f,\psi_f):M_f\oplus K_f\to X_f$ is an isomorphism.
Take an $R$-homomorphism $\varepsilon: F\to X$ with $F$ free such that $(\nu,\psi,\varepsilon):M\oplus K\oplus F\to X$ is surjective.
We have an exact sequence $0 \to L \to N \overset{(\nu,\psi,\varepsilon)}{\longrightarrow} X \to 0$ of Cohen-Macaulay $R$-modules, where $N:=M\oplus K\oplus F$.
There is a commutative diagram
$$
\begin{CD}
@. @. 0 \\
@. @. @VVV \\
0 @>>> 0 @>>> M_f\oplus K_f @>{(\nu_f,\psi_f)}>> X_f @>>> 0 \\
@. @VVV @V{
\left(
\begin{smallmatrix}
1 & 0 \\
0 & 1 \\
0 & 0
\end{smallmatrix}
\right)
}VV @| \\
0 @>>> L_f @>>> M_f\oplus K_f\oplus F_f @>{(\nu_f,\psi_f,\varepsilon_f)}>> X_f @>>> 0 \\
@. @. @VVV \\
@. @. F_f \\
@. @. @VVV \\
@. @. 0
\end{CD}
$$
with exact rows and columns.
The snake lemma implies that $L_f$ is isomorphic to $F_f$, and hence each $L_{\p_i}$ is isomorphic to the free $R_{\p_i}$-module $F_{\p_i}$.
Therefore we have $\V_R(L)\cap\Phi=\emptyset$.
Next, let us prove that $\V_R(L)$ is contained in $\V_R(M)$.
Fix any prime ideal $\p\in\V_R(L)$.
There is a commutative diagram
$$
\begin{CD}
@. @. 0 \\
@. @. @VVV \\
@. @. L_\p \\
@. @. @VVV \\
0 @>>> K_\p @>{
\left(
\begin{smallmatrix}
0 \\
1 \\
0
\end{smallmatrix}
\right)
}>> M_\p\oplus K_\p\oplus F_\p @>>> M_\p\oplus F_\p @>>> 0 \\
@. @| @V{(\nu_\p,\psi_\p,\varepsilon_\p)}VV @VVV \\
0 @>>> K_\p @>{\psi_\p}>> X_\p @>{\phi_\p}>> M_\p @>>> 0 \\
@. @. @VVV \\
@. @. 0
\end{CD}
$$
with exact rows and columns.
It follows from the snake lemma that we have an exact sequence $0 \to L_\p \to M_\p\oplus F_\p \to M_\p \to 0$ of $R_\p$-modules.
Since $L_\p$ is not a free $R_\p$-module, neither is $M_\p$.
Thus $\p$ belongs to $\V_R(M)$, as desired.
\end{proof}

The additive closure of a localization of a resolving subcategory is resolving again.

\begin{lem}\label{addres}
Let $\X$ be a resolving subcategory of $\mod R$.
Then the subcategory $\add_{R_\p}\X_\p$ of $\mod R_\p$ is resolving for every prime ideal $\p$ of $R$.
\end{lem}

\begin{proof}
The category $\add_{R_\p}\X_\p$ is closed under direct summands, and contains $R_\p$.
According to Proposition \ref{resdef}, it suffices to prove that $\add_{R_\p}\X_\p$ is closed under extensions and syzygies.

First, we verify that $\add_{R_\p}\X_\p$ is closed under extensions.
Let $0 \to L \overset{f}{\to} M \overset{g}{\to} N \to 0$ be an exact sequence of $R_\p$-modules such that $L$ and $N$ are in $\add_{R_\p}\X_\p$.
Then we have $L\oplus L'\cong X_\p$ and $N\oplus N'\cong Y_\p$ for some $R$-modules $X,Y\in\X$ and $R_\p$-modules $L',N'$.
Taking the direct sum of the identity map $L'\to L'$ (respectively, $N'\to N'$) and $f$ (respectively, $g$), we have an exact sequence $\sigma: 0 \to X_\p \to M' \to Y_\p \to 0$ of $R_\p$-modules such that $M$ is a direct summand of $M'$.
This exact sequence $\sigma$ can be regarded as an element of $\Ext_{R_\p}^1(Y_\p,X_\p)\cong\Ext_R^1(Y,X)_\p$, and we observe that it is isomorphic to $\tau_\p$ for some exact sequence $\tau: 0 \to X \to E \to Y \to 0$ of $R$-modules.
Hence $M'$ is isomorphic to $E_\p$, and $E$ belongs to $\X$ since $\X$ is closed under extensions.
Therefore $M$ belongs to $\add_{R_\p}\X_\p$.

Next, we check that $\add_{R_\p}\X_\p$ is closed under syzygies.
Let $M$ be an $R_\p$-module in $\add_{R_\p}\X_\p$.
Then $M\oplus N\cong X_\p$ for some $R_\p$-module $N$ and $X\in\X$.
Hence we get isomorphisms $\Omega_{R_\p}M\oplus\Omega_{R_\p}N\cong\Omega_{R_\p}X_\p\cong(\Omega_RX)_\p$ up to free summand.
Since $\X$ is closed under syzygies, $(\Omega_RX)_\p$ is in $\X_\p$, and so $\Omega_{R_\p}M$ is in $\add_{R_\p}\X_\p$.
\end{proof}

Let $\Phi$ be a subset of $\Spec R$.
Recall that the {\em dimension} $\dim\Phi$ of $\Phi$ is defined as the supremum of $\dim R/\p$ where $\p$ runs through all prime ideals in $\Phi$.
Hence $\dim\Phi=-\infty$ if and only if $\Phi$ is empty.
We denote by $\min\Phi$ the set of minimal elements of $\Phi$ with respect to inclusion relation.
Then note that $\dim\Phi$ is equal to the supremum of $\dim R/\p$ where $\p$ runs through all prime ideals in $\min\Phi$.
The proposition below is the essential part of the main result of this section.

\begin{prop}\label{thickcm}
Let $R$ be a Cohen-Macaulay local ring with a canonical module $\omega$.
Let $\X$ be a thick subcategory of $\CM(R)$ containing $R$ and $\omega$.
Let $M$ be a Cohen-Macaulay $R$-module.
Assume that $\kappa(\p)$ belongs to $\widetilde{\add_{R_\p}\X_\p}$ for all $\p\in\V_R(M)$.
If $\V_R(M)$ is contained in $\V_R(\X)$, then $M$ belongs to $\X$.
\end{prop}

\begin{proof}
We use induction on $\dim\V_R(M)$.

When $\dim\V_R(M)=-\infty$, we have $\V_R(M)=\emptyset$.
Hence $M$ is a free $R$-module, and $M$ belongs to $\X$.

When $\dim\V_R(M)=0$, the set $\V_R(M)$ coincides with $\{\m\}$, hence, $M$ is free on the punctured spectrum of $R$.
Corollary \ref{kore} shows that $M$ is in $\ext(\Omega^dk)$.
By assumption, $\kappa(\m)$ belongs to $\widetilde{\add_{R_\m}\X_\m}$, which means that the $R$-module $k$ belongs to $\widetilde\X$, and $\Omega^dk$ belongs to $\X$ by Proposition \ref{7}.
Since the subcategory $\X$ of $\mod R$ is closed under direct summands and extensions, $M$ belongs to $\X$.

Let us consider the case where $\dim\V_R(M)\ge 1$.
Fix a prime ideal $\p\in\min\V_R(M)$.
We check step by step that $M_\p$ satisfies the induction hypothesis.

(1) It is easy to see from the minimality of $\p$ that $\dim\V_{R_\p}(M_\p)=0<\dim\V_R(M)$.

(2) It is obvious that $\add_{R_\p}\X_\p$ is a subcategory of $\CM(R_\p)$ containing $R_\p$ and $\omega_\p$.

(3) Proposition \ref{thres1} and Lemma \ref{addres} say that $\add_{R_\p}\X_\p$ is a resolving subcategory of $\mod R_\p$.
Let $0 \to L \to M \to N \to 0$ be an exact sequence in $\CM(R_\p)$ such that $L$ and $M$ are in $\add_{R_\p}\X_\p$.
Then we have $L\oplus L'\cong X_\p$ for some $L'\in\add_{R_\p}\X_\p$ and $X\in\X$.
We obtain an exact sequence $0 \to X_\p \to M' \to N \to 0$, where $M':=M\oplus L'\in\add_{R_\p}\X_\p$.
Since $X$ is a Cohen-Macaulay $R$-module, there is an exact sequence $0 \to X \to \omega^{\oplus n} \to Y \to 0$ of Cohen-Macaulay $R$-modules.
As $X$ and $\omega$ belong to $\X$, the module $Y$ also belongs to $\X$ by the thickness of $\X$.
We make the following pushout diagram.
$$
\begin{CD}
@. 0 @. 0 \\
@. @VVV @VVV \\
0 @>>> X_\p @>>> M' @>>> N @>>> 0 \\
@. @VVV @VVV @| \\
0 @>>> \omega_\p^{\oplus n} @>>> E @>>> N @>>> 0 \\
@. @VVV @VVV \\
@. Y_\p @= Y_\p \\
@. @VVV @VVV \\
@. 0 @. 0
\end{CD}
$$
Since $N$ is a Cohen-Macaulay $R_\p$-module, we have $\Ext_{R_\p}^1(N,\omega_\p)=0$.
Hence the middle row is a split exact sequence, and we get an exact sequence $0 \to M' \to \omega_\p^{\oplus n}\oplus N \to Y_\p \to 0$.
Since $M'$ and $Y_\p$ are in $\add_{R_\p}\X_\p$ and since $\add_{R_\p}\X_\p$ is a resolving subcategory of $\mod R_\p$, the module $N$ belongs to $\add_{R_\p}\X_\p$.
Consequently, $\add_{R_\p}\X_\p$ is a thick subcategory of $\CM(R_\p)$.

(4) Our assumption implies that $\kappa(\q)$ is in $(\add_{(R_\p)_\q}(\add_{R_\p}\X_\p)_\q)^{\widetilde\quad}$ for any $\q\in\V_{R_\p}(M_\p)$.

(5) It is easily checked that $\V_{R_\p}(M_\p)$ is contained in $\V_{R_\p}(\add_{R_\p}\X_\p)$.

The above arguments (1)--(5) enable us to apply the induction hypothesis to $M_\p$, and we see that $M_\p$ belongs to $\add_{R_\p}\X_\p$ for all $\p\in\min\V_R(M)$.

Now, by virtue of Proposition \ref{lnx}, there exists an exact sequence $0 \to L \to N \to X \to 0$ of Cohen-Macaulay $R$-modules with $X\in\X$, $\V_R(L)\subseteq\V_R(M)$ and $\V_R(L)\cap\min\V_R(M)=\emptyset$ such that $M$ is a direct summand of $N$.
We have $\dim\V_R(L)<\dim\V_R(M)$.
We can apply the induction hypothesis to $L$, and see that $L$ is in $\X$.
The above exact sequence shows that $N$ is also in $\X$, and so is $M$.
\end{proof}

We recall that the {\em non-Gorenstein locus} $\ng R$ of $R$ is defined as the set of prime ideals $\p$ of $R$ such that the local ring $R_\p$ is not Gorenstein.
Let $R$ be a Cohen-Macaulay local ring.
For a subset $\Phi$ of $\Spec R$, we denote by $\V_{\CM}^{-1}(\Phi)$ the subcategory of $\CM(R)$ consisting of all Cohen-Macaulay $R$-modules $M$ such that $\V(M)$ is contained in $\Phi$.
The following theorem is the main result of this section.

\begin{thm}\label{origin}
Let $R$ be a Cohen-Macaulay local ring with a canonical module $\omega$.
Then one has the following one-to-one correspondence:
$$
\begin{CD}
\left\{
\begin{matrix}
\text{thick subcategories }\X\text{ of }\CM(R)\\
\text{containing }R,\omega\text{ with}\\
\kappa(\p)\in\widetilde{\add_{R_\p}\X_\p}\text{ for all }\p\in\V(\X)\\
\end{matrix}
\right\}
\begin{matrix}
@>{\V}>>\\
@<<{\V_{\CM}^{-1}}<
\end{matrix}
\left\{
\begin{matrix}
\text{specialization-closed subsets }\Phi\\
\text{of }\Spec R\text{ satisfying}\\
\ng R\subseteq\Phi\subseteq\Sing R
\end{matrix}
\right\}.
\end{CD}
$$
\end{thm}

\begin{proof}
Fix a thick subcategory $\X$ of $\CM(R)$ containing $R$ and $\omega$ such that $\kappa(\p)$ belongs to $\widetilde{\add_{R_\p}\X_\p}$ for every $\p\in\V(\X)$, and a specialization-closed subset $\Phi$ of $\Spec R$ with $\ng R\subseteq\Phi\subseteq\Sing R$.
We prove the theorem step by step.

(1) The set $\V(\X)$ is a specialization-closed subset of $\Spec R$ by Proposition \ref{nfsupp}(2).

(i) Let $\p$ be a prime ideal such that $R_\p$ is not Gorenstein.
Then $\omega_\p$ is not free over $R_\p$.
Since $\X$ contains $\omega$, the prime ideal $\p$ is in $\V(\X)$.
Thus $\ng R$ is contained in $\V(\X)$.

(ii) Proposition \ref{nfsupp}(1) shows that $\V(\X)$ is contained in $\Sing R$.

(2) We have $\V(\V_{\CM}^{-1}(\Phi))=\Phi$.
In fact, it is clear that $\V(\V_{\CM}^{-1}(\Phi))$ is contained in $\Phi$.
Let $\p$ be a prime ideal in $\Phi$.
Proposition \ref{variety}(6) implies that $R/\p$ belongs to $\V^{-1}(\Phi)$, and $\Omega^d(R/\p)$ belongs to $\V_{\CM}^{-1}(\Phi)$ by Proposition \ref{variety}(3).
Hence $\V(\Omega^d(R/\p))$ is contained in $\V(\V_{\CM}^{-1}(\Phi))$.
Since the local ring $R_\p$ is not regular, the $R_\p$-module $\kappa(\p)$ has infinite projective dimension.
We see from this that $\p$ is in $\V(\Omega^d(R/\p))$, and so $\p$ is in $\V(\V_{\CM}^{-1}(\Phi))$.
Thus $\Phi$ is contained in $\V(\V_{\CM}^{-1}(\Phi))$, as desired.

(3) (i) It is clear that $\V_{\CM}^{-1}(\Phi)$ contains $R$.

(ii) The category $\V^{-1}(\Phi)$ of $\mod R$ is resolving by Proposition \ref{variety}(3).
In order to see that $\V_{\CM}^{-1}(\Phi)$ is a thick subcategory of $\CM(R)$, it is enough to check that if $0 \to L \to M \to N \to 0$ is an exact sequence of Cohen-Macaulay $R$-modules such that $L,M$ belong to $\V_{\CM}^{-1}(\Phi)$, then $N$ also belongs to $\V_{\CM}^{-1}(\Phi)$.
Let $\p$ be a prime ideal in $\V(N)$.
Assume that $\p$ is not in $\Phi$.
Then $L_\p$ and $M_\p$ must be free $R_\p$-modules, which implies that $N_\p$ has projective dimension at most $1$.
The Auslander-Buchsbaum formula shows that $N_\p$ is a free $R_\p$-module, which contradicts the choice of $\p$.
Hence $\p$ is in $\Phi$, and we have that $N$ belongs to $\V_{\CM}^{-1}(\Phi)$.
Thus $\V_{\CM}^{-1}(\Phi)$ is a thick subcategory of $\CM(R)$.

(iii) Since $\Phi$ contains $\ng R$, we see that the canonical module $\omega$ is in $\V_{\CM}^{-1}(\Phi)$.

(iv) Let $\p$ be a prime ideal in $\V(\V_{\CM}^{-1}(\Phi))$.
Then the arguments in (2) prove that $\Omega_R^d(R/\p)$ belongs to $\V_{\CM}^{-1}(\Phi)$.
Hence $\Omega_{R_\p}^d\kappa(\p)$ belongs to $\add_{R_\p}(\V_{\CM}^{-1}(\Phi))_\p$, which implies that $\kappa(\p)$ is in $(\add_{R_\p}(\V_{\CM}^{-1}(\Phi))_\p)^{\widetilde\quad}$.

(4) We have $\V_{\CM}^{-1}(\V(\X))=\X$.
Indeed, it is obvious that $\X$ is contained in $\V_{\CM}^{-1}(\V(\X))$.
Let $M$ be a module in $\V_{\CM}^{-1}(\V(\X))$.
Then $\V(M)$ is contained in $\V(\X)$.
Proposition \ref{thickcm} proves that $M$ belongs to $\X$.
Thus $\V_{\CM}^{-1}(\V(\X))$ is contained in $\X$, as required.

Consequently, we obtain the one-to-one correspondence in the theorem.
\end{proof}

\begin{cor}\label{useinex}
Let $R$ be a Cohen-Macaulay local ring with a canonical module $\omega$.
Let $\X$ be a thick subcategory of $\CM(R)$ containing $R$ and $\omega$.
Then the following are equivalent:
\begin{enumerate}[\rm (1)]
\item
One has $\kappa(\p)\in\widetilde{\add_{R_\p}\X_\p}$ for all $\p\in\V(\X)$;
\item
One has $R/\p\in\widetilde\X$ for all $\p\in\V(\X)$.
\end{enumerate}
Hence one has the following one-to-one correspondence:
$$
\begin{CD}
\left\{
\begin{matrix}
\text{thick subcategories }\X\text{ of }\CM(R)\\
\text{containing }R,\omega\text{ with}\\
R/\p\in\widetilde\X\text{ for all }\p\in\V(\X)\\
\end{matrix}
\right\}
\begin{matrix}
@>{\V}>>\\
@<<{\V_{\CM}^{-1}}<
\end{matrix}
\left\{
\begin{matrix}
\text{specialization-closed subsets }\Phi\\
\text{of }\Spec R\text{ satisfying}\\
\ng R\subseteq\Phi\subseteq\Sing R
\end{matrix}
\right\}.
\end{CD}
$$
\end{cor}

\begin{proof}
(2) $\Rightarrow$ (1): Localization at $\p$ shows this implication.

(1) $\Rightarrow$ (2): According to Theorem \ref{origin}, there exists a specialization-closed subset $\Phi$ of $\Spec R$ with $\ng R\subseteq\Phi\subseteq\Sing R$ such that $\X=\V_{\CM}^{-1}(\Phi)$.
Let $\p$ be a prime ideal in $\V(\X)=\Phi$.
Then $R/\p$ belongs to $\V^{-1}(\Phi)$ by Proposition \ref{variety}(6), and hence $\Omega_R^d(R/\p)$ belongs to $\V_{\CM}^{-1}(\Phi)=\X$ by Proposition \ref{variety}(3).
This implies that $R/\p$ is in $\widetilde\X$.

The last assertion follows from Theorem \ref{origin}.
\end{proof}

Let $R$ be a Cohen-Macaulay local ring.
We recall that $R$ is said to be {\em generically Gorenstein} if $R_\p$ is a Gorenstein local ring for all $\p\in\Min R$.
(Here $\Min R$ denotes the set of minimal prime ideals of $R$.)
Also, we recall that an $R$-module $M$ is said to be {\em generically free} if $M_\p$ is a free $R_\p$-module for all $\p\in\Min R$.

\begin{ex}
Let $R$ be a generically Gorenstein, Cohen-Macaulay local ring with a canonical module $\omega$.
Put $\Phi=\Sing R\setminus\Min R$.
Then it is straightforward that $\Phi$ is a specialization-closed subset of $\Spec R$ satisfying $\ng R\subseteq\Phi\subseteq\Sing R$.
We easily observe that the subcategory $\V_{\CM}^{-1}(\Phi)$ of $\CM(R)$ consists of all Cohen-Macaulay $R$-modules that are generically free.
By Corollary \ref{useinex}, the generically free Cohen-Macaulay $R$-modules form a thick subcategory $\X$ of $\CM(R)$ containing $R$ and $\omega$ such that $R/\p$ is in $\widetilde\X$ for all $\p\in\V(\X)=\Phi$.
\end{ex}

%%%%%%%%%%%%%%%%%%%%%%%%%%%%%%%%%%%%%%%%%%%%%%%%%%%%%%%%%%%%%
%%%%%%%%%%%%%%%%%%%%%%%%%%%%%%%%%%%%%%%%%%%%%%%%%%%%%%%%%%%%%
\section{Thick and resolving subcategories and hypersurfaces}\label{tsocmoh}
%%%%%%%%%%%%%%%%%%%%%%%%%%%%%%%%%%%%%%%%%%%%%%%%%%%%%%%%%%%%%
%%%%%%%%%%%%%%%%%%%%%%%%%%%%%%%%%%%%%%%%%%%%%%%%%%%%%%%%%%%%%

In this section, we study thick subcategories of $\CM(R)$ containing $R$ when $R$ is an abstract hypersurface, and study resolving subcategories of $\mod R$ contained in $\CM(R)$ containing $\Omega^dk$ when $R$ is a Cohen-Macaulay local ring that is an abstract hypersurface on the punctured spectrum.
We shall completely classify those subcategories by corresponding them to specialization-closed subsets of $\Spec R$ contained in $\Sing R$.
First of all, we recall the definitions of a hypersurface and an abstract hypersurface.

\begin{dfn}
\begin{enumerate}[(1)]
\item
A local ring $R$ is called a {\em hypersurface} if there exist a regular local ring $S$ and an element $f$ of $S$ such that $R$ is isomorphic to $S/(f)$.
\item
A local ring $R$ is called an {\em abstract hypersurface} if there exist a complete regular local ring $S$ and an element $f$ of $S$ such that the completion $\widehat R$ of $R$ is isomorphic to $S/(f)$.
\end{enumerate}
\end{dfn}

\begin{rem}\label{hyprem}
Every hypersurface is an abstract hypersurface, while we do not know whether there exists an abstract hypersurface that is not a hypersurface.
A complete local ring is a hypersurface if and only if it is an abstract hypersurface.
In particular, for an artinian local ring, being a hypersurface is equivalent to being an abstract hypersurface.
\end{rem}

Over an abstract hypersurface $R$, a thick subcategory of $\CM(R)$ containing $R$ is nothing but a resolving subcategory of $\mod R$ contained in $\CM(R)$:

\begin{prop}\label{thres2}
Let $R$ be an abstract hypersurface, and let $\X$ be a subcategory of $\mod R$.
Then $\X$ is a thick subcategory of $\CM(R)$ containing $R$ if and only if $\X$ is a resolving subcategory of $\mod R$ contained in $\CM(R)$.
\end{prop}

\begin{proof}
We have already dealt with the `only if' part in Proposition \ref{thres1}.
Let us consider the `if' part.
Let $0 \to L \to M \to N \to 0$ be an exact sequence in $\CM(R)$.
Suppose that $L$ and $M$ belong to $\X$.
We want to show that $N$ also belongs to $\X$.
Applying Lemma \ref{zurasi}, we get an exact sequence $0 \to \Omega N \to L\oplus F \to M \to 0$ of $R$-modules with $F$ free.
As $L\oplus F$ and $M$ are in the resolving subcategory $\X$ of $\mod R$, so is $\Omega N$, and hence so is $\Omega^2N$.
We have an isomorphism $\Omega^2N\cong N$ up to free summand; see \cite[Theorem 5.1.1]{Av2}.
Hence $N$ is in $\X$.
Thus $\X$ is a thick subcategory of $\CM(R)$ containing $R$.
\end{proof}

Using Proposition \ref{thres2}, we can show the following result.

\begin{cor}\label{thresc}
Let $R$ be an abstract hypersurface.
Let $M$ be a Cohen-Macaulay $R$-module.
Then $\res M$ is the thick subcategory of $\CM(R)$ generated by $R\oplus M$.
\end{cor}

Next, we want to classify extension-closed subcategories of $\mod R$ when $R$ is an artinian hypersurface.
For this, we establish a lemma.

\begin{lem}\label{arseq}
Let $S$ be a discrete valuation ring with maximal ideal $(x)$.
Let $R=S/(x^n)$ with $n\ge 1$.
Then for each integer $1\le i\le n-1$ there exists an exact sequence
\begin{equation}\label{i-1+1}
0 \to R/(x^i) \overset{f}{\to} R/(x^{i-1})\oplus R/(x^{i+1}) \overset{g}{\to} R/(x^i) \to 0
\end{equation}
of $R$-modules, where $x^0:=1$.
\end{lem}

\begin{proof}
Existence of the exact sequences \eqref{i-1+1} is well known to experts; they are {\em Auslander-Reiten sequences} for $\mod R$.
One can directly show the existence, defining the maps $f,g$ by $f(\overline a)=\binom{\overline a}{\overline{ax}}$ and $g(\binom{\overline a}{\overline b})=\overline{ax-b}$, where $\overline{(\cdot)}$ denotes the residue class.
\end{proof}

Now we can give a classification result of extension-closed subcategories of modules over an artinian hypersurface.
There exist only four such subcategories.

\begin{prop}\label{arthyp}
Let $R$ be an artinian hypersurface.
Then all extension-closed subcategories of $\mod R$ are the empty subcategory, the zero subcategory, $\add R$ and $\mod R$.
\end{prop}

\begin{proof}
It is easily seen that there exist a discrete valuation ring $S$ with maximal ideal $(x)$ and a positive integer $n$ such that $R$ is isomorphic to $S/(x^n)$ (cf. Remark \ref{hyprem}).
Applying to $S$ the structure theorem for finitely generated modules over a principal ideal domain, we can show that $\mod R=\add_R(R\oplus R/(x)\oplus R/(x^2)\oplus\cdots\oplus R/(x^{n-1}))$.
Suppose that $\X$ is not any of the empty subcategory, the zero category and $\add R$.
Then, since $\X$ is closed under direct summands, it contains $R/(x^l)$ for some $1\le l\le n-1$.
Lemma \ref{arseq} says that $\X$ contains both $R/(x^{l-1})$ and $R/(x^{l+1})$.
An inductive argument implies that $\X$ contains $R/(x), R/(x^2), \dots, R/(x^{n-1}), R/(x^n)=R$.
Hence $\X$ coincides with $\mod R$.
\end{proof}

We state here two lemmas.
The first one is easy.

\begin{lem}\label{sop}
Let $R$ be an abstract hypersurface.
Then there exists a system of parameters $\xx$ of $R$ such that $R/(\xx)$ is an artinian hypersurface.
\end{lem}

\begin{lem}\label{resom}
Let $X$ be an $R$-module.
Let $x$ be an $R$- and $X$-regular element of $R$.
If $M$ belongs to $\res_{R/(x)}X/xX$, then $\Omega_RM$ belongs to $\res_RX$.
\end{lem}

\begin{proof}
The $R/(x)$-module $M$ is in $\res_{R/(x)}^nX/xX$ for some integer $n\ge 0$.
Let us prove by induction on $n$ that $\Omega_RM$ is in $\res_RX$.

When $n=0$, by definition $M$ is in $\add_{R/(x)}(X/xX\oplus R/(x))$, and $\Omega_RM$ is in $\add_R(\Omega_R(X/xX)\oplus R)$.
Applying Lemma \ref{zurasi} to the exact sequence $0 \to X \overset{x}{\to} X \to X/xX \to 0$, we obtain an exact sequence $0 \to \Omega_R(X/xX) \to X\oplus F \to X \to 0$ with $F$ free.
Since $\res_RX$ is resolving, $\Omega_R(X/xX)$ is in $\res_RX$, and therefore $\Omega_RM$ is in $\res_RX$.

When $n\ge 1$, there exist exact sequences of $R/(x)$-modules
\begin{align*}
0 \to S_i \to Y_i \to T_i \to 0 & \quad (1\le i\le p), \\
0 \to Y_i \to S_i \to T_i \to 0 & \quad (p+1\le i\le q)
\end{align*}
with $S_i,T_i\in\res_{R/(x)}^{n-1}(X/xX)$ for $1\le i\le q$ such that $M$ is a direct summand of $Y_1\oplus\cdots\oplus Y_q$.
Then there are exact sequences of $R$-modules
\begin{align*}
0 \to \Omega_RS_i \to \Omega_RY_i\oplus F_i \to \Omega_RT_i \to 0 & \quad (1\le i\le p), \\
0 \to \Omega_RY_i \to \Omega_RS_i\oplus F_i \to \Omega_RT_i \to 0 & \quad (p+1\le i\le q)
\end{align*}
with each $F_i$ free, and the induction hypothesis implies that $\Omega_RS_i$ and $\Omega_RT_i$ are in $\res_RX$ for $1\le i\le q$.
Since $\res X$ is resolving, $\Omega_RY_i$ is in $\res_RX$ for $1\le i\le q$, and so is $\Omega_RM$.
\end{proof}

Now we can show the following proposition.

\begin{prop}\label{hypkey}
Let $R$ be an abstract hypersurface, and let $\X$ be a resolving subcategory of $\mod R$ contained in $\CM(R)$.
Then $\kappa(\p)$ belongs to $\widetilde{\add_{R_\p}\X_\p}$ for every $\p\in\V_R(\X)$.
\end{prop}

\begin{proof}
Let $\p$ be a prime ideal in $\V_R(\X)$.
Then there exists an $R$-module $X\in\X$ such that $X_\p$ is a nonfree $R_\p$-module.
Note that $R_\p$ is an abstract hypersurface by \cite[Propositions 4.2.4(1) and 4.2.5(1); Theorem 5.1.1; Remark 8.1.1(3)]{Av2}.
Lemma \ref{sop} gives a system of parameters $\xx=x_1,\dots,x_n$ of $R_\p$ such that $R_\p/(\xx)$ is an artinian hypersurface.
Proposition \ref{arthyp} implies that all resolving subcategories of $\mod R_\p/(\xx)$ are the empty subcategory, the zero subcategory, $\add_{R_\p/(\xx)}R_\p/(\xx)$ and $\mod R_\p/(\xx)$.
Since $\X$ is contained in $\CM(R)$, the $R$-module $X$ is Cohen-Macaulay, and so is the $R_\p$-module $X_\p$.
Hence $\xx$ is an $R_\p$- and $X_\p$-sequence, and we see from \cite[Lemma 1.3.5]{BH} that $X_\p/\xx X_\p$ is a nonfree $R_\p/(\xx)$-module.
Therefore $\res_{R_\p/(\xx)}X_\p/\xx X_\p$ is a resolving subcategory of $\mod R_\p/(\xx)$ which is different from any of the empty subcategory, the zero subcategory and $\add_{R_\p/(\xx)}R_\p/(\xx)$, and thus it must coincide with $\mod R_\p/(\xx)$.
In particular, $\res_{R_\p/(\xx)}X_\p/\xx X_\p$ contains $\kappa(\p)$.
Applying Lemma \ref{resom} repeatedly, we see that $\Omega_{R_\p}^n\kappa(\p)$ belongs to $\res_{R_\p}X_\p$, which is contained in $\add_{R_\p}\X_\p$ by Lemma \ref{addres}.
We conclude that $\kappa(\p)$ belongs to $\widetilde{\add_{R_\p}\X_\p}$.
\end{proof}

Combining Theorem \ref{origin} and Proposition \ref{hypkey}, we obtain the following theorem which is a main result of this section.

\begin{thm}\label{vhyp}
Let $R$ be an abstract hypersurface.
Then one has the following one-to-one correspondence:
$$
\begin{CD}
\left\{
\begin{matrix}
\text{thick subcategories of }\CM(R)\\
\text{containing }R
\end{matrix}
\right\}
\begin{matrix}
@>{\V}>>\\
@<<{\V_{\CM}^{-1}}<
\end{matrix}
\left\{
\begin{matrix}
\text{specialization-closed subsets of }\Spec R\\
\text{contained in }\Sing R
\end{matrix}
\right\}.
\end{CD}
$$
\end{thm}

The assumption that $R$ is an abstract hypersurface in the above theorem looks too strong, but as the proposition below says, for a Cohen-Macaulay local ring having a module of complexity one, the converse of the assertion of the theorem holds.

\begin{prop}
Let $R$ be a Cohen-Macaulay local ring.
Assume there is an $R$-module of complexity one, namely, a bounded $R$-module of infinite projective dimension.
If one has the one-to-one correspondence in Theorem \ref{vhyp}, then $R$ is an abstract hypersurface.
\end{prop}

\begin{proof}
Let $\X$ be the subcategory of $\CM(R)$ consisting of all bounded Cohen-Macaulay $R$-modules.
By Example \ref{thex}(7), $\X$ is a thick subcategory of $\CM(R)$ containing $R$.
Let $M$ be a bounded $R$-module of infinite projective dimension.
Then $\Omega^dM$ is a nonfree bounded Cohen-Macaulay $R$-module.
Hence $\Omega^dM$ belongs to $\X$, and $\m$ is in $\V(\Omega^dM)$.
It follows that $\m$ is in $\V(\X)$, and we get $\V(\Omega^dk)\subseteq\{\m\}\subseteq\V(\X)$, hence $\Omega^dk$ belongs to $\V_{\CM}^{-1}(\V(\X))$.
The one-to-one correspondence in Theorem \ref{vhyp} implies $\V_{\CM}^{-1}(\V(\X))=\X$.
Thus $k$ is a bounded $R$-module.
By \cite[Remark 8.1.1(3)]{Av2}, the ring $R$ is an abstract hypersurface.
\end{proof}

Next, we consider a Cohen-Macaulay local ring which is an abstract hypersurface on the punctured spectrum.

\begin{prop}\label{lch}
Let $R$ be a Cohen-Macaulay local ring such that $R_\p$ is an abstract hypersurface for every nonmaximal prime ideal $\p$ of $R$.
Let $\X$ be a resolving subcategory of $\mod R$ which contains $\Omega^dk$ and is contained in $\CM(R)$.
Let $M$ be a Cohen-Macaulay $R$-module.
If $\V_R(M)$ is contained in $\V_R(\X)$, then $M$ belongs to $\X$.
\end{prop}

\begin{proof}
Let us prove the proposition by induction on $n:=\dim\V_R(M)$.

When $n=-\infty$, we have $\V_R(M)=\emptyset$.
This means that $M$ is free, and $M$ belongs to $\X$.

When $n=0$, the module $M$ is free on the punctured spectrum of $R$.
Corollary \ref{kore} implies $M$ is in $\ext(\Omega^dk)$.
Since $\X$ is resolving and contains $\Omega^dk$, the module $M$ is in $\X$.

Let $n>0$.
Fix a prime ideal $\p\in\min\V_R(M)$.
Then $\m$ is not in $\min\V_R(M)$, and so $\p$ is different from $\m$.
By assumption, the local ring $R_\p$ is an abstract hypersurface.
The set $\V_{R_\p}(M_\p)$ is contained in $\V_{R_\p}(\add_{R_\p}\X_\p)$.
Lemma \ref{addres} and Proposition \ref{thres2} say that $\add_{R_\p}\X_\p$ is a thick subcategory of $\CM(R_\p)$ containing $R_\p$.
By Theorem \ref{vhyp}, we have $\add_{R_\p}\X_\p=\V_{\CM}^{-1}(\V(\add_{R_\p}\X_\p))$, which contains $M_\p$.
Therefore, it follows from Proposition \ref{lnx} that there exists an exact sequence $0 \to L \to N \to X \to 0$ of Cohen-Macaulay $R$-modules such that $X$ belongs to $\X$, that $M$ is a direct summand of $N$, that $\V_R(L)$ is contained in $\V_R(M)$, and that $V_R(L)$ does not meet $\min\V_R(M)$.
Hence $L$ is a Cohen-Macaulay $R$-module with $\dim\V_R(L)<\dim\V_R(M)$ such that $\V_R(L)$ is contained in $\V_R(\X)$.
The induction hypothesis shows that $L$ belongs to $\X$.
It is seen from the above short exact sequence that $M$ also belongs to $\X$.
\end{proof}

Using the above proposition, let us prove the following classification theorem of resolving subcategories, which is a main result of this section.

\begin{thm}\label{lochp}
Let $R$ be a Cohen-Macaulay singular local ring such that $R_\p$ is an abstract hypersurface for every nonmaximal prime ideal $\p$ of $R$.
Then one has the following one-to-one correspondence:
$$
\begin{CD}
\left\{
\begin{matrix}
\text{resolving subcategories of }\\
\mod R\text{ contained in }\\
\CM(R)\text{ containing }\Omega^dk
\end{matrix}
\right\}
\begin{matrix}
@>{\V}>>\\
@<<{\V_{\CM}^{-1}}<
\end{matrix}
\left\{
\begin{matrix}
\text{nonempty specialization-closed}\\
\text{subsets of }\Spec R\\
\text{contained in }\Sing R
\end{matrix}
\right\}.
\end{CD}
$$
\end{thm}

\begin{proof}
Fix a resolving subcategory $\X$ of $\mod R$ contained in $\CM(R)$ containing $\Omega^dk$, and a nonempty specialization-closed subset $\Phi$ of $\Spec R$ contained in $\Sing R$.

(1) We see from Proposition \ref{nfsupp} that $\V(\X)$ is a specialization-closed subset of $\Spec R$ contained in $\Sing R$.

(2) The assumption that $R$ is singular implies that $\Omega^dk$ is nonfree as an $R$-module.
Hence $\m$ is in $\V(\Omega^dk)$, and $\V(\X)$ is nonempty.

(3) It follows from Proposition \ref{variety}(3) and Example \ref{resex}(3) that $\V_{\CM}^{-1}(\Phi)$ is a resolving subcategory of $\mod R$ contained in $\CM(R)$.

(4) Since $\Phi$ is nonempty, there is a prime ideal $\p\in\Phi$.
Since $\Phi$ is specialization-closed, the maximal ideal $\m$ is in $\Phi$.
Hence $\V(\Omega^dk)\subseteq\{\m\}\subseteq\Phi$, which shows $\Omega^dk$ is in $\V_{\CM}^{-1}(\Phi)$.

(5) The argument (2) in the proof of Theorem \ref{origin} proves the equality $\V(\V_{\CM}^{-1}(\Phi))=\Phi$.

(6) We have $\V_{\CM}^{-1}(\V(\X))=\X$.
Indeed, the inclusion $\V_{\CM}^{-1}(\V(\X))\supseteq\X$ is obvious, and the inclusion $\V_{\CM}^{-1}(\V(\X))\subseteq\X$ follows from Proposition \ref{lch}.
\end{proof}

As we observed in Proposition \ref{thres1}, a thick subcategory of $\CM(R)$ containing $R$ is always a resolving subcategory of $\mod R$.
Theorem \ref{lochp} yields that the converse of this statement holds for a resolving subcategory of $\mod R$ contained in $\CM(R)$ containing $\Omega^dk$ if $R$ is a Cohen-Macaulay local ring that is an abstract hypersurface on the punctured spectrum.

\begin{cor}\label{resthk}
Let $R$ be a Cohen-Macaulay local ring such that $R_\p$ is an abstract hypersurface for every nonmaximal prime ideal $\p$ of $R$.
Let $\X$ be a resolving subcategory of $\mod R$ contained in $\CM(R)$ containing $\Omega^dk$.
Then $\X$ is a thick subcategory of $\CM(R)$.
\end{cor}

\begin{proof}
If $R$ is regular, then we have $\CM(R)=\add R$, and $\X$ must coincide with $\add R$.
Let $R$ be singular.
Theorem \ref{lochp} implies $\X=\V_{\CM}^{-1}(\Phi)$, where $\Phi=\V(\X)$.
The argument (3)(ii) in the proof of Theorem \ref{origin} shows that $\V_{\CM}^{-1}(\Phi)$ is a thick subcategory of $\CM(R)$.
\end{proof}

An abstract hypersurface singular local ring and a Cohen-Macaulay singular local ring with an isolated singularity are trivial examples of a ring which satisfies the assumption of Theorem \ref{lochp}.
We end this section by putting some nontrivial examples.

\begin{ex}
Let $k$ be a field.
The following rings $R$ are Cohen-Macaulay singular local rings which are hypersurfaces on the punctured spectrums.
\begin{enumerate}[(1)]
\item
Let $R=k[[x,y,z]]/(x^2,yz)$.
Then $R$ is a $1$-dimensional local complete intersection which is neither a hypersurface nor with an isolated singularity.
All the prime ideals of $R$ are $\p=(x,y)$, $\q=(x,z)$ and $\m=(x,y,z)$.
It is easy to observe that both of the local rings $R_\p$ and $R_\q$ are hypersurfaces.
\item
Let $R=k[[x,y,z,w]]/(y^2-xz,yz-xw,z^2-yw,zw,w^2)$.
Then $R$ is a $1$-dimensional Gorenstein local ring which is neither a complete intersection nor with an isolated singularity.
All the prime ideals are $\p=(y,z,w)$ and $\m=(x,y,z,w)$.
We easily see that $R_\p$ is a hypersurface.
\item
Let $R=k[[x,y,z]]/(x^2,xy,yz)$.
Then $R$ is a $1$-dimensional Cohen-Macaulay local ring which is neither Gorenstein nor with an isolated singularity.
All the prime ideals are $\p=(x,y)$, $\q=(x,z)$ and $\m=(x,y,z)$.
We have that $R_\p$ is a hypersurface and that $R_\q$ is a field.
\end{enumerate}
\end{ex}

%%%%%%%%%%%%%%%%%%%%%%%%%%%%%%%%%%%%%%%%%%%%%%%%%%%%%%%%%%%%%
%%%%%%%%%%%%%%%%%%%%%%%%%%%%%%%%%%%%%%%%%%%%%%%%%%%%%%%%%%%%%
\section{Thick subcategories of stable Cohen-Macaulay modules}\label{tsoscm}
%%%%%%%%%%%%%%%%%%%%%%%%%%%%%%%%%%%%%%%%%%%%%%%%%%%%%%%%%%%%%
%%%%%%%%%%%%%%%%%%%%%%%%%%%%%%%%%%%%%%%%%%%%%%%%%%%%%%%%%%%%%

In this section, we make results on the stable category of Cohen-Macaulay modules from ones on the category of Cohen-Macaulay modules obtained in Sections \ref{tsocm} and \ref{tsocmoh}.
For this aim, we start by making the following definition.

\begin{dfn}\label{bar}
Let $R$ be a Cohen-Macaulay local ring.
\begin{enumerate}[(1)]
\item
For a subcategory $\X$ of $\CM(R)$, we define the category $\underline\X$ as follows:
\begin{enumerate}[(i)]
\item
$\Ob(\underline\X)=\Ob(\X)$.
\item
$\Hom_{\underline\X}(M,N)=\lhom_R(M,N)$ for $M,N\in\Ob(\underline\X)$.
\end{enumerate}
\item
For a subcategory $\Y$ of $\lCM(R)$, we define the category $\overline\Y$ as follows:
\begin{enumerate}[(i)]
\item
$\Ob(\overline\Y)=\Ob(\Y)$.
\item
$\Hom_{\overline\Y}(M,N)=\Hom_R(M,N)$ for $M,N\in\Ob(\overline\Y)$.
\end{enumerate}
\end{enumerate}
\end{dfn}

The proposition below says that classifying thick subcategories of $\lCM(R)$ is equivalent to classifying thick subcategories of $\CM(R)$ containing $R$.
It can easily be proved, and we leave the proof to the reader.

\begin{prop}\label{dotdot}
Let $R$ be a Gorenstein local ring.
Then one has the following one-to-one correspondences:
$$
\begin{CD}
\left\{
\begin{matrix}
\text{subcategories of }\CM(R)\text{ which are}\\
\text{closed under direct sums and}\\
\text{direct summands and contain }R
\end{matrix}
\right\}
@.
\begin{matrix}
@>{\underline{(\cdot)}}>>\\
@<<{\overline{(\cdot)}}<
\end{matrix}
@.
\left\{
\begin{matrix}
\text{nonempty subcategories of }\lCM(R)\\
\text{which are closed under finite}\\
\text{direct sums and direct summands}
\end{matrix}
\right\}\\
@A{\subseteq}AA @. @AA{\subseteq}A \\
\left\{
\begin{matrix}
\text{thick subcategories of }\CM(R)\\
\text{ containing }R
\end{matrix}
\right\}
@.
\begin{matrix}
@>{\underline{(\cdot)}}>>\\
@<<{\overline{(\cdot)}}<
\end{matrix}
@.
\left\{
\begin{matrix}
\text{thick subcategories}\\
\text{of }\lCM(R)
\end{matrix}
\right\}.
\end{CD}
$$
For every Cohen-Macaulay $R$-module $M$, the thick subcategory of $\CM(R)$ generated by $R\oplus M$ corresponds to the thick subcategory of $\lCM(R)$ generated by $M$.
\end{prop}

Now, we define the notion of a support for objects and subcategories of the stable category of Cohen-Macaulay modules.

\begin{dfn}
Let $R$ be a Cohen-Macaulay local ring.
\begin{enumerate}[(1)]
\item
For an object $M$ of $\lCM(R)$, we denote by $\lSupp M$ the set of prime ideals $\p$ of $R$ such that the localization $M_\p$ is not isomorphic to the zero module $0$ in the category $\lCM(R_\p)$.
We call it the {\em stable support} of $M$.
\item
For a subcategory $\Y$ of $\lCM(R)$, we denote by $\lSupp\Y$ the union of $\lSupp M$ where $M$ runs through all nonisomorphic objects in $\Y$.
We call it the {\em stable support} of $\Y$.
\item
For a subset $\Phi$ of $\Spec R$, we denote by $\lSupp^{-1}\Phi$ the subcategory of $\lCM(R)$ consisting of all objects $M\in\lCM(R)$ such that $\lSupp M$ is contained in $\Phi$.
\end{enumerate}
\end{dfn}

The notion of a stable support is essentially the same thing as that of a nonfree locus.

\begin{prop}\label{lsuppv}
Let $R$ be a Cohen-Macaulay local ring.
\begin{enumerate}[\rm (1)]
\item
Let $M$ be a Cohen-Macaulay $R$-module.
Then one has $\lSupp M=\V(M)$.
\item
Let $\X$ be a subcategory of $\CM(R)$.
Then one has $\lSupp\underline\X=\V(\X)$.
\item
Let $\Y$ be a subcategory of $\lCM(R)$.
Then one has $\lSupp\Y=\V(\overline\Y)$.
\item
Let $\Phi$ be a subset of $\Spec R$.
Then one has $\lSupp^{-1}\Phi=\underline{\V_{\CM}^{-1}(\Phi)}$.
\end{enumerate}
\end{prop}

The proof of this proposition is straightforward.

Theorem \ref{origin} and Propositions \ref{dotdot}, \ref{lsuppv} yield the result below.

\begin{thm}\label{tilde}
Let $R$ be a Gorenstein local ring.
Then one has the following one-to-one correspondence:
$$
\begin{CD}
\left\{
\begin{matrix}
\text{thick subcategories }\Y\\
\text{of }\lCM(R)\text{ with }\kappa(\p)\in\widetilde{\add_{R_\p}\overline{\Y}_\p}\\
\text{for all }\p\in\lSupp\Y
\end{matrix}
\right\}
\begin{matrix}
@>{\lSupp}>>\\
@<<{\lSupp^{-1}}<
\end{matrix}
\left\{
\begin{matrix}
\text{specialization-closed subsets}\\
\text{of }\Spec R\text{ contained in }\Sing R
\end{matrix}
\right\}.
\end{CD}
$$
\end{thm}

By Theorem \ref{vhyp} and Proposition \ref{dotdot}, we obtain the following theorem which classifies the thick subcategories of the stable category of Cohen-Macaulay modules over an abstract hypersurface.

\begin{thm}\label{tilth}
Let $R$ be an abstract hypersurface.
Then one has the following one-to-one correspondence:
$$
\begin{CD}
\left\{
\begin{matrix}
\text{thick subcategories}\\
\text{of }\lCM(R)\\
\end{matrix}
\right\}
\begin{matrix}
@>{\lSupp}>>\\
@<<{\lSupp^{-1}}<
\end{matrix}
\left\{
\begin{matrix}
\text{specialization-closed subsets of }\Spec R\\
\text{contained in }\Sing R
\end{matrix}
\right\}.
\end{CD}
$$
\end{thm}

Using Theorem \ref{lochp}, Corollary \ref{resthk}, Propositions \ref{thres1} and \ref{dotdot}, we get the following classification theorem analogous to Theorem \ref{tilth}.

\begin{thm}\label{gph}
Let $R$ be a Gorenstein singular local ring such that $R_\p$ is an abstract hypersurface for every nonmaximal prime ideal $\p$ of $R$.
Then  one has the following one-to-one correspondence:
$$
\begin{CD}
\left\{
\begin{matrix}
\text{thick subcategories of }\lCM(R)\\
\text{containing }\Omega^dk
\end{matrix}
\right\}
\begin{matrix}
@>{\lSupp}>>\\
@<<{\lSupp^{-1}}<
\end{matrix}
\left\{
\begin{matrix}
\text{nonempty specialization-closed subsets}\\
\text{of }\Spec R\text{ contained in }\Sing R
\end{matrix}
\right\}.
\end{CD}
$$
\end{thm}

Gathering Theorems \ref{vhyp}, \ref{lochp}, \ref{tilth}, \ref{gph}, Proposition \ref{thres2} and Corollary \ref{resthk} together, we obtain the following classification theorem, which is the main result of this paper.

\begin{thm}\label{main}
\begin{enumerate}[\rm (1)]
\item
Let $R$ be an abstract hypersurface.
Then one has the following one-to-one correspondences:
$$
\begin{CD}
\left\{
\begin{matrix}
\text{thick subcategories}\\
\text{of }\lCM(R)
\end{matrix}
\right\}
@.
\begin{matrix}
@>{\lSupp}>>\\
@<<{\lSupp^{-1}}<
\end{matrix}
@.
\left\{
\begin{matrix}
\text{specialization-closed subsets of }\Spec R\\
\text{contained in }\Sing R
\end{matrix}
\right\}
\\
\\
@.
\begin{matrix}
@>{\V_{\CM}^{-1}}>>\\
@<<{\V}<
\end{matrix}
@.
\left\{
\begin{matrix}
\text{thick subcategories of }\CM(R)\\
\text{containing }R
\end{matrix}
\right\}
\\
\\
@.
\begin{matrix}
@=
\end{matrix}
@.
\left\{
\begin{matrix}
\text{resolving subcategories of }\mod R\\
\text{contained in }\CM(R)
\end{matrix}
\right\}.
\end{CD}
$$
\item
Let $R$ be a Gorenstein singular local ring such that $R_\p$ is an abstract hypersurface for every nonmaximal prime ideal $\p$ of $R$.
Then one has the following one-to-one correspondences:
$$
\begin{CD}
\left\{
\begin{matrix}
\text{thick subcategories of }\lCM(R)\\
\text{containing }\Omega^dk
\end{matrix}
\right\}
@.
\begin{matrix}
@>{\lSupp}>>\\
@<<{\lSupp^{-1}}<
\end{matrix}
@.
\left\{
\begin{matrix}
\text{nonempty specialization-closed subsets}\\
\text{of }\Spec R\text{ contained in }\Sing R
\end{matrix}
\right\}
\\
\\
@.
\begin{matrix}
@>{\V_{\CM}^{-1}}>>\\
@<<{\V}<
\end{matrix}
@.
\left\{
\begin{matrix}
\text{thick subcategories of }\CM(R)\\
\text{containing }R\text{ and }\Omega^dk
\end{matrix}
\right\}
\\
\\
@.
\begin{matrix}
@=
\end{matrix}
@.
\left\{
\begin{matrix}
\text{resolving subcategories of }\mod R\\
\text{contained in }\CM(R)\text{ containing }\Omega^dk
\end{matrix}
\right\}.
\end{CD}
$$
\end{enumerate}
\end{thm}

Note that Theorem \ref{main}(2) recovers Corollary \ref{omegadk}.
Applying Theorem \ref{main}(1), we observe that over an abstract hypersurface $R$ having an isolated singularity there are only trivial resolving subcategories of $\mod R$ contained in $\CM(R)$ and thick subcategories of $\CM(R)$.

\begin{cor}\label{his}
Let $R$ be an abstract hypersurface with an isolated singularity.
\begin{enumerate}[\rm (1)]
\item
All resolving subcategories of $\mod R$ contained in $\CM(R)$ are $\add R$ and $\CM(R)$.
\item
All thick subcategories of $\lCM(R)$ are the empty subcategory, the zero subcategory, and $\lCM(R)$.
\end{enumerate}
\end{cor}

\begin{proof}
Since $R$ has an isolated singularity, the singular locus $\Sing R$ is either $\emptyset$ or $\{\m\}$.
Hence all the specialization-closed subsets of $\Spec R$ contained in $\Sing R$ are $\emptyset$ and $\Sing R$.
We have $\V_{\CM}^{-1}(\emptyset)=\add R$ and $\V_{\CM}^{-1}(\Sing R)=\CM(R)$.
Also, $\lSupp^{-1}(\emptyset)$ is the zero subcategory of $\lCM(R)$ and $\lSupp^{-1}(\Sing R)=\lCM(R)$.
The corollary now follows from Theorem \ref{main}(1).
\end{proof}

Here let us consider an example of a hypersurface which does not have an isolated singularity, and an example of a Gorenstein local ring which is not a hypersurface but a hypersurface on the punctured spectrum.

\begin{ex}\label{countex}
\begin{enumerate}[(1)]
\item
Let $R=k[[x,y]]/(x^2)$ be a one-dimensional hypersurface over a field $k$.
Then we have $\CM(R)=\add\{ R,(x),(x,y^n)\mid n\ge 1\}$ by \cite[Example (6.5)]{Y2} or \cite[Proposition 4.1]{BGS}.
Set $\p=(x)$ and $\m=(x,y)$.
We have $\Sing R=\Spec R=\{\p,\m\}$, hence all specialization-closed subsets of $\Spec R$ (contained in $\Sing R$) are $\emptyset$, $\{\m\}$ and $\Sing R$.
We have $\V_{\CM}^{-1}(\emptyset)=\add R$ and $\V_{\CM}^{-1}(\Sing R)=\CM(R)$.
The subcategory $\V_{\CM}^{-1}(\{\m\})$ of $\CM(R)$ consists of all Cohen-Macaulay modules that are free on the punctured spectrum of $R$, so it coincides with $\add\{ R,(x,y^n)\mid n\ge 1\}$.
Thus, by Theorem \ref{main}(1), all resolving subcategories of $\mod R$ contained in $\CM(R)$ are $\add R$, $\add\{ R,(x,y^n)\mid n\ge 1\}$ and $\CM(R)$.
All thick subcategories of $\lCM(R)$ are the empty subcategory, the zero subcategory, $\underline{\add\{(x,y^n)\mid n\ge 1\}}$ and $\lCM(R)$.
\item
Let $R=k[[x,y,z]]/(x^2,yz)$ be a one-dimensional complete intersection over a field $k$.
Then $R$ is neither a hypersurface nor with an isolated singularity.
All prime ideals of $R$ are $\p=(x,y)$, $\q=(x,z)$ and $\m=(x,y,z)$.
It is easy to see that $R_\p,R_\q$ are hypersurfaces.
Note that all the nonempty specialization-closed subsets of $\Spec R$ (contained in $\Sing R$) are the following four sets: $\v(\p),\v(\q),\v(\p,\q),\v(\p,\q,\m)$.
Theorem \ref{main}(2) says that there exist just four thick subcategories of $\lCM(R)$ containing $\Omega^dk$, and exist just four resolving subcategories of $\mod R$ contained in $\CM(R)$ containing $\Omega^dk$.
\end{enumerate}
\end{ex}

\begin{rem}\label{yogen}
Let $R$ be a Gorenstein local ring.
Corollary \ref{omegadk} implies that in the case where $R$ has an isolated singularity, a thick subcategory of $\lCM(R)$ coincides with $\lCM(R)$ whenever it contains $\Omega^dk$.
Example \ref{countex} especially says that this statement does not necessarily hold if one removes the assumption that $R$ has an isolated singularity.
Indeed, with the notation of Example \ref{countex}(1), $\underline{\add\{(x,y^n)\mid n\ge 1\}}$ is a thick subcategory of $\lCM(R)$ containing $\Omega^dk=\m$ which does not coincide with $\lCM(R)$.
Example \ref{countex}(2) also gives three such subcategories.
\end{rem}

%%%%%%%%%%%%%%%%%%%%%%%%%%%%%%%%%%%%%%%%%%%%%%%%%%%%%%%%%%%%%
%%%%%%%%%%%%%%%%%%%%%%%%%%%%%%%%%%%%%%%%%%%%%%%%%%%%%%%%%%%%%
\section{Applications}\label{applic}
%%%%%%%%%%%%%%%%%%%%%%%%%%%%%%%%%%%%%%%%%%%%%%%%%%%%%%%%%%%%%
%%%%%%%%%%%%%%%%%%%%%%%%%%%%%%%%%%%%%%%%%%%%%%%%%%%%%%%%%%%%%

In this section, we give some applications of our Theorem \ref{main}.
First, we have a vanishing result of homological and cohomological $\delta$-functors from the category of finitely generated modules over an abstract hypersurface.

\begin{prop}\label{delta}
Let $R$ be an abstract hypersurface and $M$ an $R$-module.
Let $\A$ be an abelian category.
\begin{enumerate}[\rm (1)]
\item
Let $T:\mod R\to\A$ be a covariant or contravariant homological $\delta$-functor with $T_i(R)=0$ for $i\gg 0$.
If there exists an $R$-module $M$ with $\pd_RM=\infty$ and $T_i(M)=0$ for $i\gg 0$, then $T_i(k)=0$ for $i\gg 0$.
\item
Let $T:\mod R\to\A$ be a covariant or contravariant cohomological $\delta$-functor with $T^i(R)=0$ for $i\gg 0$.
If there exists an $R$-module $M$ with $\pd_RM=\infty$ and $T^i(M)=0$ for $i\gg 0$, then $T^i(k)=0$ for $i\gg 0$.
\end{enumerate}
\end{prop}

\begin{proof}
(1) First of all, note that each $T_i$ preserves direct sums (cf. \cite[Proposition II.9.5]{HS2}).

We claim that for any $R$-module $N$ and any integers $n\ge 0$ and $i\gg 0$ we have
$$
T_i(\Omega^nN)\cong
\begin{cases}
T_{i+n}(N) & \text{if }T\text{ is covariant},\\
T_{i-n}(N) & \text{if }T\text{ is contravariant}.
\end{cases}
$$
Indeed, suppose $T$ is covariant.
There is an exact sequence $0 \to \Omega N \to F \to N \to 0$ with $F$ free, and from this we get an exact sequence $T_{i+1}(F) \to T_{i+1}(N) \to T_i(\Omega N) \to T_i(F)$ for $i\ge 0$.
Since $T_i(R)=0$ for $i\gg 0$, we have $T_i(F)=0$ for $i\gg 0$.
Therefore $T_{i+1}(N)\cong T_i(\Omega N)$ for $i\gg 0$.
An inductive argument yields an isomorphism $T_{i+n}(N)\cong T_i(\Omega^nN)$ for $n\ge 0$ and $i\gg 0$.
The claim in the case where $T$ is contravariant is similarly proved.

Consider the subcategory $\X$ of $\CM(R)$ consisting of all Cohen-Macaulay $R$-modules $X$ with $T_i(X)=0$ for $i\gg 0$.
Then it is easily observed that $\X$ is a thick subcategory of $\CM(R)$ containing $R$.
Since $T_i(\Omega^dM)$ is isomorphic to $T_{i+d}(M)$ (respectively, $T_{i-d}(M)$) for $i\gg 0$ if $T$ is covariant (respectively, contravariant), the nonfree Cohen-Macaulay $R$-module $\Omega^dM$ belongs to $\X$.
Hence the maximal ideal $\m$ belongs to $\V(\Omega^dM)$, which is contained in $\V(\X)$, and we have $\V(\Omega^dk)\subseteq\{\m\}\subseteq\V(\X)$.
Therefore $\Omega^dk$ belongs to $\V_{\CM}^{-1}(\V(\X))$, which coincides with $\X$ by Theorem \ref{main}.
Thus we obtain $T_i(\Omega^dk)=0$ for $i\gg 0$.
Since $T_i(\Omega^dk)$ is isomorphic to $T_{i+d}(k)$ (respectively, $T_{i-d}(k)$) for $i\gg 0$ if $T$ is covariant (respectively, contravariant), we have $T_i(k)=0$ for $i\gg 0$, as desired.

(2) An analogous argument to the proof of (1) shows this assertion.
\end{proof}

As a corollary of Proposition \ref{delta}, we obtain a vanishing result of Tor and Ext modules.

\begin{cor}\label{torext}
Let $R$ be an abstract hypersurface.
Let $M$ and $N$ be $R$-modules.
\begin{enumerate}[\rm (1)]
\item
One has $\Tor_i^R(M,N)=0$ for $i\gg 0$ if and only if either $\pd_RM<\infty$ or $\pd_RN<\infty$.
\item
One has $\Ext_R^i(M,N)=0$ for $i\gg 0$ if and only if either $\pd_RM<\infty$ or $\id_RN<\infty$.
\end{enumerate}
\end{cor}

\begin{proof}
The `if' parts are trivial.
In the following, we consider the `only if' parts.

(1) Assume that $M$ has infinite projective dimension.
The functors $\Tor_i^R(-,N):\mod R\to\mod R$ form a covariant homological $\delta$-functor and $\Tor_i^R(R,N)=0$ for $i>0$.
Hence Proposition \ref{delta}(1) implies that $\Tor_i^R(k,N)=0$ for $i\gg 0$, which means that $N$ has finite projective dimension.

(2) Assume that $M$ has infinite projective dimension.
The functors $\Ext_R^i(-,N):\mod R\to\mod R$ form a contravariant cohomological $\delta$-functor and $\Ext_R^i(R,N)=0$ for $i>0$.
Hence Proposition \ref{delta}(2) implies that $\Ext_R^i(k,N)=0$ for $i\gg 0$, which means that $N$ has finite injective dimension.
\end{proof}

The first assertion of Corollary \ref{torext} gives another proof of a theorem of Huneke and Wiegand \cite[Theorem 1.9]{HW}.

\begin{cor}[Huneke-Wiegand]\label{hwii+1}
Let $R$ be an abstract hypersurface.
Let $M$ and $N$ be $R$-modules.
If $\Tor_i^R(M,N)=\Tor_{i+1}^R(M,N)=0$ for some $i\ge 0$, then either $M$ or $N$ has finite projective dimension.
\end{cor}

\begin{proof}
Similarly to the beginning part of the proof of \cite[Theorem 1.9]{HW}, we have $\Tor_j^R(M,N)=0$ for all $j\ge i$.
Now apply Corollary \ref{torext}(1) to get the conclusion.
\end{proof}

\begin{rem}
Several generalizations of Corollaries \ref{torext}(1), \ref{hwii+1} to complete intersections have been obtained by Jorgensen \cite{J1,J2}, Miller \cite{Mi} and Avramov and Buchweitz \cite{AvB}.
\end{rem}

The assertions of Corollary \ref{torext} do not necessarily hold if the ring $R$ is not an abstract hypersurface.

\begin{ex}
Let $k$ be a field.
Consider the artinian complete intersection local ring $R=k[[x,y]]/(x^2,y^2)$.
Then we can easily verify $\Tor_i^R(R/(x),R/(y))=0$ and $\Ext_R^i(R/(x),R/(y))=0$ for all $i>0$.
But both $R/(x)$ and $R/(y)$ have infinite projective dimension, and infinite injective dimension by \cite[Exercise 3.1.25]{BH}.
\end{ex}

Let {\bf H} be a property of local rings.
Let $\hdim_R$ be a numerical invariant for $R$-modules satisfying the following conditions.
\begin{enumerate}[(1)]
\item
$\hdim_R R<\infty$.
\item
Let $M$ be an $R$-module and $N$ a direct summand of $M$.
If $\hdim_R M<\infty$, then $\hdim_R N<\infty$.
\item
Let $0 \to L \to M \to N \to 0$ be an exact sequence of $R$-modules.
\begin{enumerate}[(i)]
\item
If $\hdim_R L<\infty$ and $\hdim_R M<\infty$, then $\hdim_R N<\infty$.
\item
If $\hdim_R L<\infty$ and $\hdim_R N<\infty$, then $\hdim_R M<\infty$.
\item
If $\hdim_R M<\infty$ and $\hdim_R N<\infty$, then $\hdim_R L<\infty$.
\end{enumerate}
\item
The following are equivalent:
\begin{enumerate}[(i)]
\item
$R$ satisfies {\bf H};
\item
$\hdim_RM<\infty$ for any $R$-module $M$;
\item
$\hdim_R k<\infty$.
\end{enumerate}
\end{enumerate}

The conditions (1) and (3) imply the following condition:
\begin{enumerate}[(5)]
\item
Let $M$ be an $R$-module.
If $\pd_RM<\infty$, then $\hdim_RM<\infty$.
\end{enumerate}
Indeed, let $M$ be an $R$-module with $\pd_RM<\infty$.
Then there is an exact sequence $0 \to F_n \to F_{n-1} \to \cdots \to F_1 \to F_0 \to M \to 0$ of $R$-modules with each $F_i$ free.
The conditions (1) and (3)(ii) imply that $\hdim_RF_i<\infty$ for all $0\le i\le n$.
Decomposing the above exact sequences into short exact sequences and applying (3)(i), we have $\hdim_RM<\infty$.

We call such a numerical invariant a {\em homological dimension}.
A lot of homological dimensions are known.
For example, projective dimension $\pd_R$, Gorenstein dimension $\Gdim_R$ (cf. \cite{AB,C}) and Cohen-Macaulay dimension $\CMdim_R$ (cf. \cite{Ge}) coincide with $\hdim_R$ where ${\bf H}$ is regular, Gorenstein and Cohen-Macaulay, respectively.
Several other examples of a homological dimension can be found in \cite{Av}.

A lot of studies of homological dimensions have been done so far.
For each homological dimension $\hdim_R$, investigating $R$-modules $M$ with $\hdim_RM<\infty$ but $\pd_RM=\infty$ is one of the most important problems in the studies of homological dimensions.
In this sense, the following proposition says that an abstract hypersurface does not admit a proper homological dimension.

\begin{prop}
With the above notation, let $R$ be an abstract hypersurface not satisfying {\bf H}.
Let $M$ be an $R$-module.
Then $\hdim_RM<\infty$ if and only if $\pd_RM<\infty$.
\end{prop}

\begin{proof}
The condition (5) says that $\pd_RM<\infty$ implies $\hdim_RM<\infty$.
Conversely, assume $\hdim_RM<\infty$.
Let $\X$ be the subcategory of $\CM(R)$ consisting of all Cohen-Macaulay $R$-modules $X$ satisfying $\hdim_RX<\infty$.
It follows from the conditions (1), (2) and (3) that $\X$ is a thick subcategory of $\CM(R)$ containing $R$.
Theorem \ref{main} yields $\X=\V_{\CM}^{-1}(\V(\X))$.
Suppose that $\pd_RM=\infty$.
Then $\Omega^dM$ is a nonfree Cohen-Macaulay $R$-module.
We have an exact sequence $0 \to \Omega^dM \to F_{d-1} \to \cdots \to F_0 \to M \to 0$ of $R$-modules such that $F_i$ is free for $0\le i\le d-1$.
Decomposing this into short exact sequences and using the conditions (1) and (3), we see that $\Omega^dM$ belongs to $\X$.
Hence the maximal ideal $\m$ of $R$ is in $\V(\X)$, and we obtain $\V(\Omega^dk)\subseteq\{\m\}\subseteq\V(\X)$.
Therefore $\Omega^dk$ belongs to $\V_{\CM}^{-1}(\V(\X))=\X$, namely, $\hdim_R(\Omega^dk)<\infty$.
There is an exact sequence $0 \to \Omega^dk \to G_{d-1} \to \cdots \to G_1 \to G_0 \to k \to 0$ of $R$-modules with each $G_i$ free.
Decomposing this into short exact sequences and using the conditions (1) and (3), we get $\hdim_Rk<\infty$.
Thus the condition (4) implies that $R$ satisfies the property {\bf H}, which contradicts our assumption.
Consequently, we must have $\pd_RM<\infty$.
\end{proof}

Next, we want to consider in the one-to-one correspondences obtained in Theorem \ref{main}, what resolving subcategories of $\mod R$ and what thick subcategories of $\lCM(R)$ the closed subsets of $\Spec R$ correspond to.
In general, the structure of the subcategories corresponding to a closed subset is as follows.

\begin{rem}
Let $W$ be a closed subset of $\Spec R$ contained in $\Sing R$, and let $I$ be a defining ideal of $W$.
Then the subcategory $\V^{-1}(W)$ of $\mod R$ consists of all $R$-modules $X$ such that $X_f$ is a projective $R_f$-module for all $f\in I$.
If $I$ is generated by $f_1,\dots,f_n$, then $\V^{-1}(W)$ consists of all $R$-modules $X$ such that each $X_{f_i}$ is a projective $R_{f_i}$-module.
When $R$ is Cohen-Macaulay, restricting to $\CM(R)$, we see that $\V_{\CM}^{-1}(W)$ is the subcategory of $\CM(R)$ consisting of all Cohen-Macaulay $R$-modules $X$ such that $X_f$ is a projective $R_f$-module for all $f\in I$, equivalently, such that each $X_{f_i}$ is a projective $R_{f_i}$-module.
\end{rem}

As the following proposition says, the subcategories of $\CM(R)$ and $\lCM(R)$ corresponding to a closed subset of $\Spec R$ are relatively ``small.''

\begin{prop}
Let $R$ be an abstract hypersurface.
Then one has the following one-to-one correspondences:
$$
\begin{CD}
\left\{
\begin{matrix}
\text{thick subcategories of }\lCM(R)\\
\text{generated by one object}\\
\end{matrix}
\right\}
@. @>{\subseteq}>> @.
\left\{
\begin{matrix}
\text{thick subcategories}\\
\text{of }\lCM(R)\\
\end{matrix}
\right\}
\\
\begin{matrix}
@V{\lSupp}VV @AA{\lSupp^{-1}}A
\end{matrix}
@. @. @.
\begin{matrix}
@V{\lSupp}VV @AA{\lSupp^{-1}}A
\end{matrix}
\\
\left\{
\begin{matrix}
\text{closed subsets of }\Spec R\\
\text{contained in }\Sing R
\end{matrix}
\right\}
@. @>{\subseteq}>> @.
\left\{
\begin{matrix}
\text{specialization-closed subsets of }\Spec R\\
\text{contained in }\Sing R
\end{matrix}
\right\}
\\
\begin{matrix}
@V{\V_{\CM}^{-1}}VV @AA{\V}A
\end{matrix}
@. @. @.
\begin{matrix}
@V{\V_{\CM}^{-1}}VV @AA{\V}A
\end{matrix}
\\
\left\{
\begin{matrix}
\text{resolving subcategories of }\mod R\\
\text{generated by one object in }\CM(R)
\end{matrix}
\right\}
@. @>{\subseteq}>> @.
\left\{
\begin{matrix}
\text{resolving subcategories of }\mod R\\
\text{contained in }\CM(R)
\end{matrix}
\right\}.
\end{CD}
$$
\end{prop}

\begin{proof}
We establish a couple of claims.
\setcounter{claim}{0}

\begin{claim}
Let $\X$ be a resolving subcategory of $\mod R$ contained in $\CM(R)$.
Then the following are equivalent:
\begin{enumerate}[\rm (1)]
\item
$\X=\res M$ for some $M\in\CM(R)$;
\item
$\V(\X)$ is a closed subset of $\Spec R$.
\end{enumerate}
\end{claim}

\begin{cpf}
(1) $\Rightarrow$ (2): We have $\V(\X)=\V(M)$ by Proposition \ref{nfsupp}(3), which is a closed subset of $\Spec R$ by Proposition \ref{nfsupp}(2).

(2) $\Rightarrow$ (1): Taking the irreducible decomposition of the closed subset $\V(\X)$, we have $\V(\X)=\v(\p_1)\cup\cdots\cup\v(\p_n)$, where $\p_1,\dots,\p_n$ are prime ideals of $R$.
Then each $\p_i$ is in $\V(\X)$, and there is a Cohen-Macaulay $R$-module $M_i\in\X$ such that $\p_i$ is in $\V(M_i)$.
Set $M=M_1\oplus\cdots\oplus M_n$.
Then $M$ belongs to $\X$.
Let $\p$ be a prime ideal in $\V(\X)$.
Then $\p$ contains some $\p_l$, and we have $\p_l\in\V(M_l)\subseteq\V(M)$.
Since $\V(M)$ is (specialization-)closed, $\p$ is also in $\V(M)$.
It follows that $\V(\X)$ coincides with $\V(M)$, which is equal to $\V(\res M)$ by Proposition \ref{nfsupp}(3).
Theorem \ref{main} shows that $\X$ coincides with $\res M$.
\qed
\end{cpf}

\begin{claim}
Let $\Y$ be a nonempty thick subcategory of $\lCM(R)$.
The following are equivalent:
\begin{enumerate}[\rm (1)]
\item
$\Y$ is generated by one object (as a thick subcategory of $\lCM(R)$);
\item
$\lSupp\Y$ is a closed subset of $\Spec R$.
\end{enumerate}
\end{claim}

\begin{cpf}
This follows from Propositions \ref{dotdot}, \ref{lsuppv}(3), Corollary \ref{thresc} and Claim 1.
\qed
\end{cpf}

Combining the above two claims and Theorem \ref{main}, we obtain the desired one-to-one correspondences.
\end{proof}

%%%%%%%%%%%%%%%%%%%%%%%%%%%%%%%%%%%%%%%%%%%%%%%%%%%%%%%%%%%%%
%%%%%%%%%%%%%%%%%%%%%%%%%%%%%%%%%%%%%%%%%%%%%%%%%%%%%%%%%%%%%

\begin{ac}
The author would like to give his deep gratitude to Srikanth Iyengar and the knowledgeable referee, who read the paper carefully and gave an important information on earlier literature and a lot of helpful comments and suggestions.
The author is greatly indebted to Amnon Neeman, Yuji Yoshino, Osamu Iyama, Kiriko Kato, Tokuji Araya and Kei-ichiro Iima for a lot of valuable remarks.
The author also thanks very much the participants of the Seminar on Commutative Algebra at Meiji University which is organized by Shiro Goto, for their kind advices.
\end{ac}

%%%%%%%%%%%%%%%%%%%%%%%%%%%%%%%%%%%%%%%%%%%%%%%%%%%%%%%%%%%%%
%%%%%%%%%%%%%%%%%%%%%%%%%%%%%%%%%%%%%%%%%%%%%%%%%%%%%%%%%%%%%

\end{document}